\documentclass[11pt,xcolor=dvipsnames,svgnames,table,reqno]{amsart}

\usepackage[utf8]{inputenc}
\usepackage{graphicx}
\usepackage{tipa}
\graphicspath{ {./images/} }
\usepackage{hyperref}
\usepackage{amssymb} 
\usepackage{comment}
\usepackage{amsthm} 
\usepackage{mathtools}
\usepackage{nccmath}
\usepackage{physics} 
\usepackage{tipa} 
\usepackage{tikz-cd} 
\usepackage{yfonts}
\usepackage{enumerate} 
\usepackage[T1]{fontenc} 
\usepackage[inline]{enumitem} 
\usepackage{xspace} 
\usepackage{commath}
\usepackage{tikz-cd}
\usepackage{color} 
\usepackage{pgfplots}
\pgfplotsset{compat=1.18}
\usepackage{hyperref} 
\usepackage{cleveref}

\usepackage[left=1.5in, right=1.5in, bottom=1.25in,
height=8in]{geometry}
\usepackage{wrapfig}

\def    \R      {{\mathbb R}}
\def \Z {{\mathbb Z}}

\renewcommand{\epsilon}{\varepsilon}
\newtheorem{maintheorem}{Theorem} 
\newtheorem{maincorollary}{Corollary} 
\newtheorem{theorem}{Theorem}[section] 
\newtheorem{corollary}[theorem]{Corollary} 
\newtheorem{lemma}[theorem]{Lemma}
\newtheorem{proposition}[theorem]{Proposition}

\newcommand{\bb}[1]{\mathbb{#1}}

\theoremstyle{definition}

\newtheorem{definition}[theorem]{Definition}

\theoremstyle{remark}
\newtheorem{remark}[theorem]{Remark}
\newtheoremstyle{italichead} 
  {3pt}   
  {3pt}   
  {\normalfont}  
  {}      
  {\itshape}     
  {.}     
  { }     
  {}      
\newtheorem{ex}[theorem]{Example}
\newtheorem{rmk}[theorem]{Remark}

\AtEndDocument{\bigskip{\footnotesize
 \textsc{Rafael Fernandes, UC Santa Cruz, Santa Cruz, California.} \par  
  \textit{E-mail}:  \texttt{rfernan9@ucsc.edu}\\

 \textsc{Joao Pering Jr, Stony Brook University, NY. } \par  
  \textit{E-mail}:  \texttt{joao.pering@stonybrook.edu}
  }}

\title{On the growth rate of Reeb orbits on star-shaped hypersurfaces}
\date{}
\author{Rafael Fernandes and Joao Pering}

\begin{document}

\keywords{Fiberwise star-shaped hypersurfaces, Reeb flows, growth of periodic orbits, symplectic homology, spectral invariants, string topology, symplectically degenerate maximum}

\thanks{The first author was partially supported by the NSF grant DMS-2304206 and the second author was supported by the Luxembourg National Research Fund (FNR) SYMPG 19513846, NSF award DMS-220330 and the Simons Foundation.}

\maketitle

\setcounter{tocdepth}{2}
\begin{abstract}

In this article, we study the growth rate of Reeb orbits on fiberwise star-shaped hypersurfaces in the cotangent bundle of a closed manifold. We prove that under a suitable topological condition on the base manifold the Reeb flow on any such hypersurface carries infinitely many simple closed orbits. Moreover, the number of
simple Reeb orbits with period at most $T$ grows at least like the prime
numbers, that is, like $T/\log(T)$. The topological condition we assume is the existence of a non-nilpotent class in the homology of the free loop space of the manifold, with respect to the Chas–Sullivan product, lying in a connected component associated to a non-torsion class in the first homology of the manifold. In particular, for any Riemannian metric on a manifold satisfying such a topological condition, the number of geometrically distinct closed geodesics with length at
most $l$ grows at least like $l/\log(l)$. We also prove, using symplectic
homology, that if a Liouville domain of dimension at least $4$ with vanishing first Chern class admits a Reeb symplectically degenerate maximum representing a non-torsion first homology class of the domain, then the number of simple Reeb
orbits with period at most $T$ grows at least like $T/\log(T)$. 
\end{abstract}

\tableofcontents

\section{Introduction and main results}

\subsection{Introduction} 

In this paper, we further investigate the general question of how topology influences dynamics. We show that the existence of a special homology class of the free loop space of a closed manifold forces a growth rate, with respect to the period, for Reeb orbits on the boundary of fiberwise star-shaped hypersurfaces in its cotangent bundle.

In \cite{bangert1984closed}, Bangert and Hingston prove that for any closed Riemannian manifold $(M,g)$ of dimension at least $2$ whose fundamental group is abelian and contains an element of infinite order, the number of geometrically distinct closed geodesics of length at most $T$ grows at least like the prime numbers, that is, $T/\log(T)$. 

Their result can be interpreted as follows. Let $M$ be a closed Riemannian
manifold, and consider the geodesic flow restricted to the unit cotangent bundle; its integral curves are lifts of unit-speed geodesics. The
length of a closed geodesic corresponds to the period of its lift, so the number of
geometrically distinct closed geodesics of length at most $T$ equals the number
of geometrically distinct closed orbits of the geodesic flow on the unit sphere
bundle with period at most $T$. The unit sphere bundle is a hypersurface of
contact type in the cotangent bundle, and the geodesic flow on it is in fact the
Reeb flow. This motivates the question whether the same growth rate for simple
Reeb orbits holds for more general hypersurfaces of contact type in a cotangent
bundle, not necessarily arising from a Riemannian metric.

There has been some work on the growth of Reeb orbits, as a function of the
period, for fiberwise star-shaped hypersurfaces in cotangent bundles. For
instance, \cite{macarini2011positive, heistercamp2011weinstein} study, among
other things, the topological entropy of the Reeb flow on such hypersurfaces.
In \cite{macarini2011positive}, Macarini and Schlenk show that if the based loop
space of a closed manifold $M$ is sufficiently ``complicated'', then the Reeb
flow on any fiberwise star-shaped hypersurface in $T^*M$ has positive
topological entropy. In particular, if $M$ is a closed surface other than the sphere, torus,
projective plane, or Klein bottle, then the Reeb flow on any fiberwise
star-shaped hypersurface in $T^*M$ has positive topological entropy. Moreover,
it follows from results of Katok in \cite{katok1980lyapunov} and \cite[Theorem 4.1]{katok1982entropy}, or alternatively from Lima and Sarig \cite[Theorem~1.1]{lima2019symbolic}, that the number
of Reeb orbits grows exponentially with respect to the period.

The first result in the direction we are interested in appeared in
\cite{pellegrini2022bangert}. Pellegrini proved that if $M$ is a closed
manifold of dimension at least two
admitting a homologically non-trivial $S^1$-action,
then for any non-degenerate fiberwise star-shaped hypersurface in $T^*M$ the
number of simple Reeb orbits with period at most $T$ grows at least like
$\log(T)$. Here, a \textit{homologically non-trivial $S^1$-action} means a continuous circle action whose orbits represent a non-torsion homology class.

Their proof is based on a \textit{pinching argument}, introduced in
\cite{macarini2011positive}. This yields inequalities relating the spectral
invariants of Floer homology classes for a quadratic at infinity Hamiltonian on
$T^*M$ (whose level set is the hypersurface) to those of a kinetic quadratic
Hamiltonian associated to a Riemannian metric. The argument exploits a
filtration-preserving Viterbo isomorphism between Floer homology and the
homology of the free loop space, allowing them to perform estimates using the
energy functional. The underlying identification between these
groups goes back to \cite{viterbo2003functors,viterbo1999functors}. Finally, the
non-degeneracy assumption in \cite{pellegrini2022bangert} rules out the presence
of complicated orbits, so that an index argument allows them to geometrically
distinguish the orbits arising from the relevant Floer homology classes.

Instead of assuming the existence of a homologically non-trivial $S^1$-action, we assume the
existence of a non-nilpotent homology class of a connected component of the
free loop space corresponding to a non-torsion class in $H_1(M;\Z)$. This is a
weaker condition, and in Example \ref{ex: no S1 action} we present manifolds
satisfying our hypothesis but admitting no non-trivial $S^1$-action. We also remark that the mere existence of a non-nilpotent free loop space
homology class is not sufficient even to guarantee infinitely many Reeb orbits,
as illustrated by the Katok-Ziller Finsler metrics on $S^n$, which have only
finitely many closed geodesics; see \cite{katok1973ergodic, ziller1983geometry}
and \cite{cohen2003loop}. This justifies our additional assumption that the
non-nilpotent homology class lies in a connected component of the free loop space
corresponding to a non-torsion first homology class of the base manifold.

In our approach to the growth of Reeb orbits in Theorem \ref{thm: main_result}, we use the symplectic homology of
the domain bounded by the hypersurface, together with its canonical filtration, as in Section \ref{subsec: cont and sym homology}. In this setting, Viterbo's theorem provides an algebra isomorphism which is not
filtration preserving in general; see Section
\ref{subsec: SH of a cotangent bdle} and \cite{abbondandolo2006floer,abouzaid2013symplectic}. The spectral invariants of
the iterates of a non-nilpotent symplectic homology class then lead to Reeb
orbits contributing to the growth estimate. In particular, our argument avoids
choosing a Riemannian metric and implementing a pinching argument. 

Moreover, we do not assume that the fiberwise star-shaped hypersurface is
non-degenerate. In the degenerate setting, highly degenerate Reeb orbits may
appear. These are known as \textit{symplectically degenerate maxima}, or simply
SDMs; see Definition \ref{def: SDM}. In the presence of an SDM, the standard action-index arguments used to
geometrically distinguish orbits break down, and hence iterates of the SDM
could potentially create redundancy in orbit counts. Nonetheless, it is known
that the dynamics in the presence of a SDM exhibits special features.

In this paper, we study the role played by SDMs in the Reeb dynamics of the boundary of a Liouville domain. The importance
of SDMs goes back to \cite{hingston2009subharmonic, ginzburg2010conley}. In
\cite{ginzburg2010conley}, Ginzburg proved that on a symplectically aspherical
closed manifold, the presence of a SDM for a Hamiltonian diffeomorphism implies
the existence of infinitely many simple periodic orbits (which implies the degenerate case of the Conley conjecture). SDMs have since played an important role in existence
results; see, for instance,
\cite{ginzburg2010conley, ginzburg2009action, hein2012conley, atallah2024number, ginzburggurel2010conley}
and references therein. In these works, the periodic orbits under consideration
are contractible, whereas in our setting we need to understand
\textit{non-contractible Reeb SDMs}. Much less is known in the non-contractible case;
see \cite{ginzburg2013closed, ginzburg2015conley}. In \cite{ginzburg2013closed},
the authors prove, in particular, that if the boundary of a Liouville domain has
a simple SDM with action $c$ and mean index $\Delta$, then for every $k\in\Z$ sufficiently large
there exists a Reeb orbit with action slightly larger than $kc$ and mean index equal to $k\Delta+1$.

We prove an analogous statement using symplectic homology instead, which is convenient for our purposes. Our proof follows the same general line of ideas (first
introduced in \cite{ginzburg2010conley}), but it involves its own fundamental
difficulties and differs in important aspects from previous approaches. Beyond simple Reeb SDMs, we prove that the presence of a not necessarily simple Reeb SDM satisfying a certain topological condition implies a growth rate of geometrically distinct Reeb orbits with respect to the period; see Theorem \ref{thm: SDM generation on a covering}. In particular, the Reeb flow has infinitely many simple orbits. This is, to the best of our knowledge, the first result relating not necessarily simple nor contractile Reeb SDMs to a growth rate of the Reeb flow.

As a final remark, we point out the results of McLean in \cite{mclean2012local} and Hryniewicz and Macarini in \cite{hryniewicz2015local}. Among other things, they show, using different variants of Floer homology, that if the rank of the homology of the free loop space of a closed manifold $M$ is unbounded as a
function of the degree, then any fiberwise star-shaped hypersurface in the
cotangent bundle $T^*M$ carries infinitely many simple Reeb orbits. However, no
statement is made about the growth of the number of these orbits with respect to period.
This situation occurs, for example, when $M$ is simply connected
and its cohomology ring has at least two generators over $\mathbb{Q}$; see
\cite{vigue1976homology}. The case considered here is of a different nature:
we work with manifolds whose first Betti number is nonzero and for which the
homology of the free loop space admits a non-nilpotent element in a connected
component corresponding to a non-torsion class in $H_1(M;\Z)$.

\subsection{Main results} Let $M$ be a closed smooth manifold of dimension $n$. A \textit{fiberwise star-shaped
hypersurface} $\Sigma$ in its cotangent bundle is a hypersurface transverse to the radial vector field that is the boundary of a smooth
fiberwise star-shaped domain $W\subset T^*M$, that is, a domain such that
$W\cap T_x^*M$ is star-shaped with respect to $0_x\in T_x^*M$ for every
$x\in M$. In particular, $W$ contains the zero section, i.e., $0_M\subset W$. Moreover, since $M$ is compact, $\Sigma$ is compact as well. Recall that
$\Sigma$ carries a natural contact form induced by the Liouville form on the
cotangent bundle, so we can consider its Reeb flow; see Section \ref{subsec: SH of a cotangent bdle}. In this paper, except when explicitly stated otherwise, all homology groups are taken with coefficients in $\Z_2$. 

Let $\Lambda M$ be the free loop space of $M$, and consider its homology $H_*(\Lambda M)$ endowed with the Chas--Sullivan
product. For a class $\beta\in H_*(\Lambda M)$ and $m\in \mathbb{N}$, we denote
by $\beta^m$ the $m$-fold product of $\beta$. The connected components of $\Lambda M$ are in bijection with conjugacy classes of $\pi_1(M)$. 

For a homology class $\eta\in H_1(M;\Z)$, we
denote by $\Lambda_\eta M$ the subspace of loops representing $\eta$. This is a
union of connected components of $\Lambda M$. Moreover, we say that a homology
class $\eta\in H_1(M;\Z)$ is \textit{primitive} if it is not divisible by any
integer distinct from $\pm 1$. For each $T>0$, we define $N_{\Sigma}^\eta(T)$
to be the number of simple Reeb orbits on $\Sigma$ with period at most $T$
representing a homology class in $\mathbb{N}\eta$. An element
$\beta\in H_*(\Lambda_\eta M)$ is called a \textit{$H_1(M)$-infinite non-nilpotent class}
if $\beta^m\neq 0$ for all $m\in \mathbb{N}_{>0}$ and some $\eta\in H_1(M;\Z)$ that is non-torsion. Given $\beta \in H_*(\Lambda_\eta M)$ a $H_1(M)$-infinite non-nilpotent class, we denote by $\eta(\beta)\in H_1(M;\Z)$ the homology class $\eta$.

We are now ready to state the main result of this paper.

\begin{maintheorem}\label{thm: main_result}
Let $M$ be a closed smooth manifold of dimension at least $2$. If there exists
an $H_1(M)$-infinite non-nilpotent class $\beta$ with $\eta(\beta) = \eta$, then for any smooth fiberwise star-shaped hypersurface
$\Sigma\subset T^*M$,
\begin{equation}\label{eq:growth}
\liminf_{T\to\infty} N_{\Sigma}^{\eta}(T)\,\frac{\log(T)}{T} > 0,
\end{equation}
\end{maintheorem}

\begin{maincorollary}
In the setting of Theorem \ref{thm: main_result}, any fiberwise star-shaped
hypersurface in $T^*M$ carries infinitely many simple Reeb orbits representing
classes in $\mathbb{N}\eta$.
\end{maincorollary}

Among the manifolds satisfying the hypotheses of Theorem \ref{thm: main_result}
are:
$S^1\times M$, where $M$ is any closed manifold of dimension at least $1$;
products $M\times P$, where either $M$ or $P$ satisfies the hypotheses of Theorem
\ref{thm: main_result}; $H$-spaces with infinite fundamental group;
principal $S^1$-bundles $M$ over a base $B$ with $\pi_2(B)=0$; manifolds $M$
admitting a homologically non-trivial $S^1$-action, that is, a circle action whose orbits
represent a non-torsion first homology class; and closed manifolds that are
$K(G,1)$, where the group $G$ has an element in its center with infinite order in the abelianization of $G$. We return in Section \ref{sec: string topology and examp} to the study string topology and derive conditions on a manifold that ensure the existence of $H_1(M)$-infinite non-nilpotent classes.

We also point out that the result of Bangert and Hingston in
\cite{bangert1984closed} on the growth rate of closed geodesics on a closed
Riemannian manifold $(M,g)$ of dimension at least $2$ requires the fundamental
group $\pi_1(M)$ to be abelian. It is not clear how this assumption relates to
the hypotheses of Theorem \ref{thm: main_result}. However, in Section
\ref{sec: string topology and examp} we provide classes of manifolds with
non-abelian fundamental group that nevertheless satisfy the hypotheses of
Theorem \ref{thm: main_result} (in particular, Examples \ref{ex: mapping torus}
and \ref{ex: no S1 action}). Thus, the following corollary is meaningful.

\begin{maincorollary}
    Let $M$ be a closed smooth manifold of dimension $\geq 2$ that admits an $H_1(M)$-infinite non-nilpotent class $\beta$ with $\eta(\beta) = \eta$. For any Riemannian metric $g$, let $N^\eta_g(l)$ denote the number of geometrically distinct closed geodesics with length at most $l$ representing a homology class in $\mathbb{N}\eta$. Then  
    $$\liminf_{l\rightarrow \infty} N^{\eta}_g(l) \frac{\log(l)}{l} >0.$$
\end{maincorollary}

There are other results concerning the growth rate of closed geodesics lying in
iterates of a free homotopy class; see Ballmann and Ziller
\cite{ballmann1981closed}. Their approach requires the presence of a nontrivial finite-order element in the fundamental group (up to conjugacy), and thus applies, for instance, to manifolds with finite, nontrivial fundamental group. Moreover, the corresponding growth result is established only for
generic Riemannian metrics. See also
\cite{bangert1980closed, gromov2000three, shelukhin2023remark, rademacher2022second, contreras2025closed, allais2020growth, liu2023growth, rademacher2022closed}
for further results on closed geodesics and homotopy groups.

We also point out \cite{goresky2009loop}, where Goresky and Hingston study the
Chas-Sullivan product in the context of the homology of the energy-filtered
free loop space of a Riemannian manifold and derive consequences for closed
geodesics. In addition, in \cite{hingston2013resonance} a resonance
relation is proved for Finsler metrics on $S^n$, and non-nilpotent free loop
space homology classes play an interesting role.

Now we briefly discuss the proof of Theorem \ref{thm: main_result}. The
Chas-Sullivan product preserves the $H_1(M;\Z)$-splitting of the free loop
space
\begin{equation}\label{eq: splitting of free loop space}
    \Lambda M = \bigsqcup_{\zeta\in H_1(M;\Z)}\Lambda_\zeta M,
\end{equation}
at the level of homology. More precisely, for $\eta_1,\eta_2\in H_1(M;\Z)$,
the product is a map
\[H_*(\Lambda_{\eta_1} M)\otimes H_*(\Lambda_{\eta_2} M)\to H_{*-n}(\Lambda_{\eta_1+\eta_2} M).\]
For an $H_1(M)$-infinite non-nilpotent class $\beta$ of degree $k$, $\eta=\eta(\beta)$, it follows that
$\beta^m\in H_{mk-(m-1)n}(\Lambda_{m\eta}M)$ is a nonzero homology class for every
$m\in\mathbb{N}_{>0}$. Let $W\subset~T^*M$ be the Liouville subdomain bounded by a
fiberwise star-shaped hypersurface $\Sigma\subset T^*M$. Then Viterbo's theorem
provides an isomorphism of $\Z_2$-algebras
\begin{equation}\label{eq: vit iso}
    H_*(\Lambda M)\xrightarrow{\cong} SH_*(W),
\end{equation}
which preserves the corresponding $H_1(M;\Z)$-splitting from
\eqref{eq: splitting of free loop space}, where the right-hand side is endowed
with the pair-of-pants product (see Subsections
\ref{subsec: cont and sym homology} and \ref{subsec: SH of a cotangent bdle}, and
\cite{abbondandolo2010floer}).

For each $m\in\mathbb{N}$, consider the spectral invariant $c_{\beta^m}(W)$,
where $\beta^m$ is viewed as a class in the symplectic homology $SH_*(W)$,
together with an action selector, i.e., a Reeb orbit $x_m$ on $\Sigma$ representing the homology class $m\eta$, whose action is $c_{\beta^m}(W)$ and such that
$SH_n(x_m)\neq 0$ (see Section \ref{sec: spectral invariants}). The sublinearity property of spectral invariants implies
\begin{equation*}
    c_{\beta^m}(W) \leq m\, c_{\beta}(W).
\end{equation*}

To conclude, we need to determine whether the orbits $x_m$, or their primitives, are geometrically distinct. We restrict
our attention to $x_p$, for all is a sufficiently large prime numbers $p$. It turns out
that the $x_p$ need not be geometrically distinct for all sufficiently large
primes $p$, and may in fact all be iterates of a single orbit $\gamma$. In this
case, the Reeb orbit $\gamma$ turns out to be a symplectically degenerate maximum, or shortly, a Reeb SDM (see Definition
\ref{def: SDM}) representing a non-torsion homology class. We then apply the
following result, which is of independent interest. In what follows we assume that $c_1(W)=0$.

\begin{maintheorem} \label{thm: SDM generation on a covering}
Let $\gamma$ be an isolated not necessarily simple Reeb SDM orbit on the boundary $\Sigma$ of a Liouville domain $(W,\lambda)$, of dimension at least $4$, representing a non-torsion homology class $\eta = [\gamma]  \in H_1(W;\Z)$. Then
\begin{equation} \label{eq:growth in corollary}
    \liminf_{T\rightarrow \infty} N_\Sigma^{\eta}(T) \frac{\log(T)}{T} > 0.
\end{equation}
\end{maintheorem}
The key point of the theorem above is that the Reeb SDM does not have to be simple. However, it must represent a non-torsion homology class in the Liouville domain; cf. Corollary \ref{cor: Sdm implies inf many}. Although the proofs are completely different, Theorem \ref{thm: SDM generation on a covering} may be viewed as a Reeb version of the main Theorem in \cite{gurel2013non}; see also \cite[Proposition II]{hingston1993growth} for a similar result concerning the growth \eqref{eq:growth in corollary} in the presence of a special orbit of the geodesic flow for an orientable surface. It follows from the following technical result proved in Section \ref{section: proof of main theorems}; cf. \cite{ginzburg2010conley,ginzburg2013closed}.

\begin{maintheorem} \label{thm: SDM gene orbits}
Let $\gamma$ be a simple Reeb SDM orbit on the boundary of a Liouville domain $(W,\lambda)$ of dimension at least $4$, representing a free homotopy class of loops $ \mathfrak{f}$ of $W$ with action $\mathcal{A}(\gamma)= c$, and mean index $ \Delta(\gamma) = \Delta$. Then for every $\epsilon>0$, there exists $k_0$ such that 
\begin{equation} \label{eq:sym hom non zero}
    SH_{k\Delta + n+1}^{(kc + \delta_k, kc + \epsilon)}(W,\mathfrak{f}^k) \neq 0, \text{ for all } k>k_0 \text{ and some } \delta_k \text{ with } 0<\delta_k< \epsilon. 
\end{equation} 
\end{maintheorem}

We note that it is essential for the proof of Theorem \ref{thm: SDM gene orbits} that the Reeb SDM is simple. Although expected, it is not clear whether the theorem remains true for non-simple Reeb SDMs, and a different argument would likely be needed for a proof. The following corollary is a direct consequence of the above theorem; see \cite[Corollary~4.3]{ginzburg2015conley} for a proof.

\begin{maincorollary} \label{cor: Sdm implies inf many}
If a Liouville domain $(W,\lambda)$ as in the setting of Theorem \ref{thm: SDM gene orbits} admits a simple Reeb SDM, then the Reeb flow has infinitely many simple periodic orbits.
\end{maincorollary}

It is conjectured that the presence of an iterated Reeb SDM implies the existence of infinitely many simple Reeb orbits. Theorem \ref{thm: SDM generation on a covering} can be seen as a step towards the proof of this conjecture.
\begin{remark}
 We note that the growth rate \eqref{eq:growth} remains true for any Liouville
domain $(W,\lambda)$ with $c_1(W)=0$, provided there exists a non-nilpotent
symplectic homology class in $SH_*(W;\eta)$ with respect to the pair-of-pants
product, where $\eta\in H_1(W;\Z)$ is non-torsion. This is clear from the proof
idea of Theorem \ref{thm: main_result} outlined above, since the free loop space
homology is used only through Viterbo's isomorphism to guarantee the existence
of such a symplectic homology class.
\end{remark}

\noindent{\bf Acknowledgments:}
The authors thank Viktor Ginzburg and Mark McLean, their respective advisors, for the support during this project. We also thank Pedro Salomão and Shuhao Li for useful discussions, and the organizers of the conference \textit{New developments in symplectic geometry}, in Seoul National University, where part of this work was carried out.

\section{Preliminaries}
Let $(W,\lambda)$ be a Liouville domain with boundary $\partial W=N$,
and let $\alpha=\lambda|_N$ be the induced contact form. We denote by
$\widehat{W}$ the completion of $(W,\lambda)$, defined by
\[\widehat{W}= W \cup_N \bigl(N\times [1,\infty)\bigr),\]
equipped with the Liouville form $\widehat{\lambda}$, which is given by
$\widehat{\lambda}=r\alpha$ on the cylindrical end, where $r$ is the coordinate
on $[1,\infty)$, and by $\widehat{\lambda}=\lambda$ on $W$. Throughout this
paper, we assume for simplicity that $c_1(TW)=0$.

For the contact boundary $(N,\alpha)$, we denote by $R_{\alpha}$ the Reeb vector field, defined by
$$\text{d}\alpha(R_\alpha,\cdot)=0 \text{ and } \alpha(R_\alpha)\equiv 1.$$ A \textit{Reeb orbit} is a closed orbit of $R_{\alpha}$. More precisely, it is
a map $\gamma\colon \R/T\Z \to N$, for some $T>0$, such that
$\gamma'(t)=R_{\alpha}(\gamma(t))$. We say that a Reeb orbit is
\textit{simple} (or \textit{prime}) if it is an embedding in $N$; otherwise,
we say that the orbit is \textit{iterated}. Note that any Reeb orbit is either
simple or an $m$-fold iterate of a simple Reeb orbit, for some
$m\in\mathbb{N}$.

The number $T>0$ is called the \textit{period} of $\gamma$, and the \textit{minimal
period} is the period of the underlying simple Reeb orbit whose iterate equals
$\gamma$. We denote the period by $\mathcal{A}(\gamma)$ and refer to it as the
\textit{action} (or \textit{period}) of $\gamma$. We also denote by
$\mathcal{S}(\alpha)$ the \textit{spectrum} of $\alpha$, i.e., the set of
actions (periods) of all Reeb orbits. Note that the action of a Reeb orbit
$\gamma$ is also given by
\[\mathcal{A}(\gamma)=\int_{\R/T\Z}\gamma^*\alpha.\]

We are mainly interested in the case where $W\subset T^*M$ is a smooth
fiberwise star-shaped domain with respect to the zero section and all domains are assumed to be smooth. For such
domains, we have $c_1(TW)=0$. In what follows, we recall the definition of
symplectic homology and collect the facts that will be used later. Our goal is
to fix notation and record the statements needed in the proofs. For further
details, see \cite{cieliebak2018symplectic, bourgeois2009symplectic} and
references therein. Lastly, we recall that, as mentioned in the introduction, all homology groups are taken with coefficients in $\Z_2$ throughout this paper, unless otherwise stated.

\subsection{Hamiltonians and action function} 
Let $H\colon \widehat{W}\times S^1\to \R$ be a Hamiltonian on $\widehat{W}$. We
denote by $X_H$ its Hamiltonian vector field, defined by
\[
\omega(X_{H_t},\cdot)=-\,dH_t.
\]
Recall that the \textit{action functional} associated to $H$ is
\begin{equation}\label{eq: action functional}
\mathcal{A}_H(x)=\int_x \widehat{\lambda}-\int_{S^1} H\bigl(x(t)\bigr)\,dt,
\qquad x\in C^\infty(S^1,\widehat{W}).
\end{equation}
We denote by $\mathcal{P}(H)$ the set of critical points of $\mathcal{A}_H$,
i.e., the $1$-periodic orbits of $H$, and by $\mathcal{S}(H)$ the
\textit{spectrum} of $H$, i.e., the set of critical values of
\eqref{eq: action functional} (equivalently, the set of actions of the
$1$-periodic orbits of $H$).

We say that $H$ is \textit{linear at infinity} if there exist constants
$a,c\in\R$ such that $H(r,t)=ar-c$ on $N\times [r_0,\infty)$ for some $r_0>1$.
The constant $a$ is called the \emph{slope} of $H$, and we write
$\operatorname{slope}(H):=a$.

We consider the class of time-independent Hamiltonians $H\colon \widehat{W}\to\R$
that are constant on $W$ and depend only on the radial coordinate $r$ on
the cylindrical end $N\times[1,\infty)$, i.e.,
\[H|_{N\times[1,\infty)} = h(r),\]
where $h$ satisfies the following conditions:
\begin{itemize}\label{eq: admissible Ham}
    \item \textit{Monotonicity:} $h'(r)\ge 0$ for $r\ge 1$;
    \item \textit{Convexity:} $h''(r)\ge 0$ for all $r\ge 1$, and $h''(r)>0$ on
    $(1,r_{\max})$ for some $r_{\max}>1$ (depending on $h$);
    \item \textit{Linear at infinity:} there exist constants $a,c\in\R$ such
    that $h(r)=ar-c$ for $r\ge r_{\max}$.
\end{itemize}
Throughout, we assume $\operatorname{slope}(H)\notin \mathcal{S}(\alpha)$
whenever necessary. We refer to $H$ as \textit{admissible} if $H|_W<0$, and as
\textit{semi-admissible} if $H|_W\equiv 0$. Note that for any admissible
Hamiltonian $H$, the Hamiltonian $H - H|_{W}$ is semi-admissible. Hence, the
action spectra of these two Hamiltonians agree up to a shift (and
likewise for filtered Floer homology).

For a (semi-)admissible Hamiltonian $H$, the Hamiltonian vector field is given by
\[
X_H = h'(r)\,R_{\alpha},
\]
where $R_{\alpha}$ is the Reeb vector field. In this case, a Hamiltonian orbit
corresponds to a Reeb orbit. More precisely, given a $T$-periodic Reeb orbit
$\gamma$ with $T<\operatorname{slope}(H)$, the curve
\[
\tilde{\gamma}(t)=(\gamma(Tt),r), \qquad t\in \R/\Z,
\]
is a $1$-periodic orbit of $H$, where $h'(r)=T$. In particular,
$\tilde{\gamma}$ lies in the shell $1<r<r_{\max}$. Observe that a Hamiltonian
which is linear at infinity has periodic orbits only in a compact subset of
$\widehat{W}$, provided that $\operatorname{slope}(H)\notin \mathcal{S}(\alpha)$.

For $H$ and $\tilde\gamma(t)$ as above, the action is given by
$$\mathcal{A}_H(\tilde\gamma) = \int_0^1 r(x(t))\alpha(z'(t))dt - \int_0^1 h(r(x(t)))dt = rh'(r) - h(r).$$
Since it depends only on the radial coordinate, we write
$$A_h: [1,\infty) \rightarrow [0,\infty), \ A_h(r) = rh'(r) - h(r).$$
Note that $A_h$ is monotone increasing, because
$$A_h'(r) = rh''(r) + h'(r) - h'(r) = rh''(r) \geq 0.$$
We also consider the following modification of the action function, called the \textit{reparametrization function}
\begin{equation} \label{eq: a_h function}
a_h = A_h \circ (h')^{-1} : [0,a] \rightarrow [0,c].
\end{equation}
This function turns periods of Reeb orbits into the actions of the corresponding
Hamiltonian orbits, i.e., $a_h(\mathcal{S}(\alpha) \cap (0,a]) = \mathcal{S}(H) \cap (0,c]$. 

There are four properties of the reparametrization function $a_h$ that
will be useful throughout the paper; see \cite{cineli2026invariant,fernandes2024barcode}
for further details.
\begin{enumerate}
\item[(i)] The reparametrization function is bi-Lipschitz, i.e., for
$0\le t\le t'$
\begin{equation}\label{eq: bilipchitz}
t'-t \le a_h(t')-a_h(t)\le r_{\max}(H)\,(t'-t).
\end{equation}

\item[(ii)] If $H$ and $K$ are semi-admissible Hamiltonians with $H\le K$, then
\[
a_h \ge a_k \quad \text{on } [0,\operatorname{slope}(H)].
\]

\item[(iii)] For all $t\in[0,\operatorname{slope}(H)]$,
\begin{equation}\label{eq:inequality action function}
t \le a_h(t).
\end{equation}

\item[(iv)] For all $t\in[0,\operatorname{slope}(H)]$,
\begin{equation}\label{eq: A_h(t)convergetot}
a_{s h}(t)\to t
\end{equation}
as $s\to +\infty$.
\end{enumerate}

To prove \eqref{eq: bilipchitz}, note that for $t=h'(r)$ and $t'=h'(r')$, since
$h'$ is nondecreasing, we have $1\le r\le r'$. Moreover,
\[a_h(t')-a_h(t)=
\bigl(sh'(s)-h(s)\bigr)\Big|_{r}^{r'}
=\int_{r}^{r'}\bigl(sh'(s)-h(s)\bigr)'\,ds
=\int_{r}^{r'} s h''(s)\,ds \ge t'-t.
\]
Furthermore, by the convexity of $h$,
\[a_h(t')-a_h(t)
= \int_{r}^{r'} s h''(s)\,ds
\le r_{\max}\int_{r}^{r'} h''(s)\,ds
= r_{\max}(h)\bigl(h'(r')-h'(r)\bigr).\]

In order to prove \eqref{eq:inequality action function} and \eqref{eq: A_h(t)convergetot}, notice that for $t=h'(r)$, minus the action $-a_h(t)$ is the intersection of the tangent line to $h$ at $(r,h(r))$ with the ordinate. If $H \leq K$, and $h'(r_0) =t=k'(r_1)$, then the tangent lines to $h$ and $k$ at $(r_0,h(r_0))$ and $(r_1,k(r_1))$ respectively are parallel since they have the same slope, but the tangent line to $k$ is above the one to $h$ since $h \leq k$, and therefore $- a_k(t) \geq - a_h(t)$. Since $h$ is strictly convex in $(1,r_{\max}(H))$, the tangent line to $h$ at any point $(r,h(r))$ with $r > 1$ intersect the abscissa in $(1,\infty)$, so it is on the right to the line $r \rightarrow tr -t$, which is the line that passes through $1$ with slope $t$. The intersection of this line with the ordinate is $-t$, therefore $ -t \geq - a_h(t)$. 

\begin{center}
    \begin{tikzpicture} \label{functionA_H}
        \begin{axis}[xmin = -0.1, xmax = 3, axis lines = left, xlabel = $r$ , xtick = {1,2}, xticklabels = {$1$, $r_{\max}$}, xticklabel style={above left} ,ymin = -1.5, ymax = 3, ytick = \empty, axis lines = middle]
        \addplot[color = black, samples = 50, domain = 1:2 ]{(x-1)^2};
        \addplot[domain =2:3, color = black, samples = 50]{2*(x - 2) + 1}node[right,pos = 0.7]{$H$};
         \addplot[domain =-1:3, color = black, samples = 50]{(x - 3/2) + 1/4}; 
        \addplot[color  
        = black, samples = 50, domain = 1:2 ]{2*(x-1)^2};
        \addplot[domain = 2:3, color = black, samples = 50]{4*(x - 2) + 2}node[right,pos = 0.1]{$K$};
        \addplot[domain =-1:3, color = black, samples = 50]{1*(x -5/4) + 1/8}; 
        \addplot[domain = -1:3, color = black]{x-1}; 
         \node[pin = -1: {$-a_{h}(t)$}] at (axis cs:0, -5/4) {};
         \node[pin = 3: {$-a_{k}(t)$}] at (axis cs: 0,-9/8) {};
         \node[pin = 80: {$-t$}] at (axis cs:0, -1) {};
        \end{axis}
        \node[below] at (current bounding box.south) {Figure \ref{functionA_H}};
    \end{tikzpicture}
\end{center}

\subsection{Mean and Conley-Zehnder indices}  \label{subsection: indexes}
In this subsection, we fix notation and recall some facts about the mean index and the Conley-Zehnder index. Our conventions are closest to those in
\cite{ginzburg2020lusternik}. For further background and details, see
\cite{gutt2014generalized, abbondandolo2001morse, salamon1999lectures} and the references therein.

We normalize the Conley-Zehnder index $\mu_{cz}$ by requiring that the origin, viewed as a periodic orbit of a small non-degenerate quadratic form $Q$ on $\R^{2n}$, has index $sgn(Q)/2$, where the signature $sgn(Q)$ is the number of positive eigenvalues minus the number of negative eigenvalues of $Q$. In particular, the Conley--Zehnder index of a non-degenerate critical point of a $C^2$-small Hamiltonian on $\R^{2n}$ equals $n- \mu_{\mathrm{Morse}}$, where
$\mu_{\mathrm{Morse}}$ is the Morse index. We take the opportunity to emphasize that the unit in $SH_*(W)$ (defined later) lies in degree $n$, i.e., $1\in SH_n(W)$. Moreover, it is represented by a local minimum of a Morse function on $W$.

Recall that the \textit{lower} and \textit{upper Conley-Zehnder indices} are defined by 
\begin{equation} \label{eq: upper and lower cz index}
    \mu_- (\Phi):= \liminf
_{\hat{\Phi}\rightarrow\Phi} \mu_{cz}(\hat{\Phi}) \text{ and } \mu_+ (\Phi):= \limsup
_{\hat{\Phi}\rightarrow\Phi} \mu_{cz}(\hat{\Phi}),
\end{equation}
where both limits are taken over non-degenerate paths of symplectic matrices $\hat{\Phi}$ converging to $\Phi$. One can check that $\mu_{cz}(\Phi) = \mu_{\pm}(\Phi)$ whenever $\Phi$ is a non-degenerate path. By construction, $\mu_{\pm}$ are, respectively, upper and lower semi-continuous extension of $\mu_{cz}(\Phi)$.

The mean index of $\Phi \in \widetilde{Sp}(2n)$ is defined as
\begin{equation} \label{eq:mean index def}
    \Delta(\Phi) = \lim_{k\rightarrow \infty} \frac{\mu_{\pm}(\Phi^k)}{k}.
\end{equation}
This is the unique, up to normalization,  continuous quasi-morphism
$$\Delta : \widetilde{Sp}(2n) \rightarrow \R,$$
see, e.g., \cite{salamon1992morse,ginzburg2020lusternik,cineli2026invariant} for more details. It follows from the definition \eqref{eq:mean index def} that the mean index is homogeneous, i.e.,
\begin{equation} \label{eq: mean index homogeneous}
    \Delta(\Phi^k) = k\Delta(\Phi),\ k\in\Z,
\end{equation}
and satisfies
\begin{equation} \label{eq: mean index and conley zehnder index}
    \Delta(\Phi) - n \leq \mu_-(\Phi) \leq \mu_+(\Phi) \leq \Delta(\Phi) + n.
\end{equation}
See \cite{salamon1992morse} for a proof of \eqref{eq: mean index and conley zehnder index}. 

We recall the definition of the grading for Hamiltonian and Reeb orbits. Let
$(W,\lambda)$ be a Liouville domain with $c_1(TW)=0$, and fix a trivialization of the anti-canonical complex line bundle of $W$.

For a one-periodic orbit $x$ of a Hamiltonian $H$, the tangent bundle along $x$
has an induced trivialization (well defined up to homotopy). Using this
trivialization, we associate to $x$ a path of symplectic matrices
$\Phi_x\in \widetilde{Sp}(2n)$ corresponding to the linearized Hamiltonian flow
$d\phi_H^t(x(0))$. We then define the indices $\mu_{cz}(x)$, $\mu_\pm(x)$, and
$\Delta(x)$ by
\[\mu_{cz}(x)=\mu_{cz}(\Phi_x), \qquad \mu_\pm(x)=\mu_\pm(\Phi_x), \qquad
\Delta(x)=\Delta(\Phi_x).\]
Similarly, for a Reeb orbit $\gamma$ on the contact boundary of a Liouville domain $(N,\xi=\text{ker}(\alpha))$, choose a trivialization of the top
exterior power of $\xi$ and consider the induced
trivialization of $\xi$ along $\gamma$. This determines a path of symplectic
matrices $\Phi_\gamma$ associated to the linearized transverse flow
$d\phi_{R_\alpha}^t(\gamma(0))\big|_{\xi}$ and the indices of $\gamma$ are then defined as before.

\begin{rmk}
For a cotangent bundle $\widehat W=T^*M$, there are special choices of
trivializations of the anti-canonical bundle, and hence of $TT^*M$ along
orbits; see \cite{abbondandolo2006floer}. Throughout, we work with these
specific trivializations for a cotangent bundle.
\end{rmk}

\subsection{Floer homology}

We briefly recall the construction of filtered Floer homology. Along this subsection, $H$ is a fixed linear at infinity non-degenerate Hamiltonian. In particular, $H:\widehat{W}\times S^1 \rightarrow\R$ is time-dependent and $\text{slope}(H)\notin \mathcal{S}(\alpha)$. 

Denote by $\mathcal{J}(W)$ the space of \textit{admissible almost complex
structures} on $\widehat{W}$, i.e., the set of $S^1$-dependent $d\lambda$-compatible
almost complex structures $J_t$ on $\widehat{W}$ such that there exists a compact
subset $K\subset \widehat{W}$ with:

\begin{itemize}
    \item \textit{cylindrical at infinity:} $J$ preserves the contact distribution
    $\xi$ and satisfies $J(r\partial_r)=R_\alpha$ on $\widehat{W}\setminus K$;
    \item $J_t$ is $S^1$-independent on $\widehat{W}\setminus K$.
\end{itemize}
For $J\in \mathcal{J}(W)$, a \textit{Floer cylinder} is a smooth map
$u\colon \R\times S^1\to \widehat{W}$ solving \textit{Floer's equation}
\begin{equation}\label{eq: floer eq}
\partial_s u + J(u)\bigl(\partial_t u - X_H(u)\bigr)=0,
\end{equation}
where $s\in\R$ and $t\in S^1$. For $x_\pm\in \mathcal{P}(H)$, we define the
\textit{moduli space of connecting Floer cylinders} by
\begin{align}\label{eq: moduli}
 \mathcal{M}(x_-,x_+,H,J) =\Bigl\{u\colon \R\times S^1\to \widehat{W}\ \Big|\ u \text{ satisfies }
\eqref{eq: floer eq}\ \text{and}\ \lim_{s\to\pm\infty}u(s,\cdot)=x_\pm\Bigr\}\Big/\R,
\end{align}
where $\R$ acts by translation in the $s$-coordinate. For generic
$J\in\mathcal{J}(W)$, these moduli spaces are cut out transversely, and hence
are smooth manifolds of dimension
\[\dim \mathcal{M}(x_-,x_+,H,J)=\mu_{cz}(x_+) - \mu_{cz}(x_-) - 1.\]
We refer to such generic $J$ as \textit{regular}, and to $(H,J)$ as a
\textit{regular pair}.

For any $a,b\notin \mathcal{S}(H)$ with $a<b$, the \textit{filtered Floer chain
complexes} are defined as $\Z_2$-modules by
\[CF^b_*(H,J)=\bigoplus_{\substack{\mu_{cz}(x)=*\\ \mathcal{A}_H(x)<b}} \Z_2\cdot x,
\qquad
CF^{(a,b)}_*(H,J)=CF^b_*(H,J)/CF^a_*(H,J),\]
where the generators $x$ are $1$-periodic orbits of $H$. The differential
$\partial\colon CF^b_*(H,J)\to CF^b_{*-1}(H,J)$ acts on generators by
\[\partial x_+ = \sum_{\mu_{cz}(x_-)=\mu_{cz}(x_+)-1}
\#_2 \mathcal{M}(x_-,x_+,H,J)\, x_-,\]
where $\#_2 \mathcal{M}(x_-,x_+,H,J)$ denotes the mod $2$ count of elements in
the moduli space. The differential satisfies $\partial\circ\partial=0$, and
thus we obtain a chain complex $(CF^b_*(H,J),\partial)$ whose homology is the
\textit{filtered Floer homology} $HF^b_*(H,J)$. Note that $\partial$ decreases
both the action $\mathcal{A}_H$ and the Conley-Zehnder index $\mu_{cz}$. In
particular, $(CF^{(a,b)}_*(H,J),\partial)$ is a chain complex whose homology is denoted by $HF^{(a,b)}_*(H,J)$.

For $a<b<c$ with $a,b,c\notin \mathcal{S}(H)$, the short exact sequence
\[0 \rightarrow CF^{(a,b)}_*(H,J) \rightarrow CF^{(a,c)}_*(H,J) \rightarrow
CF^{(b,c)}_*(H,J) \rightarrow 0\]
induces a long exact sequence in homology,
\begin{equation}\label{eq: long exa seq for Hamiltonian}
\cdots \rightarrow HF_*^{(a,b)}(H,J) \rightarrow HF_*^{(a,c)}(H,J) \rightarrow
HF_*^{(b,c)}(H,J) \rightarrow HF_{*-1}^{(a,b)}(H,J) \rightarrow \cdots .
\end{equation}

Floer homology admits a decomposition with respect to $H_1(W;\Z)$,
which we write as
\begin{equation}\label{eq: H_1 splitting of HF}
HF^{(a,b)}(H,J)=\bigoplus_{\eta\in H_1(\widehat{W};\Z)} HF^{(a,b)}(H,J;\eta),
\end{equation}
where the $\eta$-factor takes into account only $1$-periodic orbits in
$\mathcal{P}(H)$ representing the homology class $\eta$. An analogous
decomposition can be made in terms of free homotopy classes of loops in $\widehat{W}$.

\subsection{Continuation maps and symplectic homology} \label{subsec: cont and sym homology}
We closely follow \cite[Section~4.4]{cieliebak2018symplectic}. Let $H_\pm$ be
non-degenerate Hamiltonians on $\widehat{W}$ that are linear at infinity and
satisfy $H_-\le H_+$, and let $J_\pm\in\mathcal{J}(W)$ be such that
$(H_\pm,J_\pm)$ are regular pairs. In particular,
$\operatorname{slope}(H_-)\le \operatorname{slope}(H_+)$. Consider a homotopy
of Hamiltonians $H_s$ such that $H_s=H_+$ for $s\ll 0$ and $H_s=H_-$ for
$s\gg 0$, with $H_s$ linear at infinity and $\partial_s H_s\le 0$. An
interpolating almost complex structure $J_s$ between $J_+$ and $J_-$ then
induces a chain map
\begin{equation}\label{eq: chain continuation map}
CF^{(a,b)}_*(H_-,J_-)\to CF^{(a,b)}_*(H_+,J_+),
\end{equation}
unique up to chain homotopy, defined by counting solutions of a parametrized
version of Floer's equation \eqref{eq: floer eq}. Then the induced map on filtered Floer homology,
\begin{equation}\label{eq: continuation map}
HF^{(a,b)}_*(H_-,J_-)\to HF^{(a,b)}_*(H_+,J_+),
\end{equation}
does not depend on the choice of interpolation data and is called the
\textit{continuation map}.

We make two remarks. First, a parametrized Floer cylinder counted in
\eqref{eq: chain continuation map} has non-increasing energy along itself, due to the monotonicity assumption $\partial_s H_s\le 0$. In particular, the filtered continuation map is well defined. Second, the same monotonicity assumption is crucial for a
no-escape lemma for Floer solutions, which in turn implies compactness of the
relevant moduli spaces entering \eqref{eq: continuation map}.

Continuation maps are functorial with respect to concatenation of the
interpolation data. One immediate consequence is that $HF^{(a,b)}(H,J)$ does
not depend on the choice of $J\in\mathcal{J}(W)$, and henceforth we omit $J$
from the notation.

Filtered Floer homology extends to a degenerate Hamiltonian $H$ that is linear
at infinity, with $\operatorname{slope}(H)\notin \mathcal{S}(\alpha)$ and
$a,b\notin \mathcal{S}(H)$, as follows. Let $\tilde H$ be a non-degenerate,
linear-at-infinity, $C^2$-small perturbation of $H$ with
$\operatorname{slope}(\tilde H)=\operatorname{slope}(H)$. We define
$HF^{(a,b)}(H)$ to be $HF^{(a,b)}(\tilde H)$ for such a perturbation $\tilde H$. By functoriality of continuation maps, $HF^{(a,b)}(H)$ is independent of the choice of $\tilde H$. Moreover, continuation maps are still defined for pairs $(H_-,H_+)$ of linear at infinity Hamiltonians as above satisfying $H_-\le H_+$.

For any $a<a'$ and $b<b'$ not in $\mathcal{S}(H)$, there are
``inclusion'' maps
\[HF^{(a,b)}_*(H)\to HF^{(a',b')}_*(H),\]
defined, for non-degenerate Hamiltonians linear at infinity, by inclusion and
quotients of the appropriate action-filtered subcomplexes, and then extended
to (semi-)~admissible Hamiltonians via suitable perturbations. These maps are
compatible with continuation maps and allow us to extend the definition of
filtered Floer homology for a (semi-)admissible Hamiltonian $H$ with
$\operatorname{slope}(H)\notin \mathcal{S}(\alpha)$ when $a$ or $b$ lies in
$\mathcal{S}(H)$. We do so by setting
\[HF^{(a,b)}_*(H)=\varinjlim_{\substack{a<a'\\ b<b'}} HF^{(a',b')}_*(H),\]
where $a',b'\notin \mathcal{S}(H)$. Once again, there continuation maps
for these homologies and they satisfy the same formal properties.

It is convenient for us to remove the condition
$\operatorname{slope}(H)\notin \mathcal{S}(\alpha)$ for the sake of cleaner
statements. We do this by setting
\[HF^t(H)=\varinjlim_{\tilde{H}\le H} HF^t(\tilde{H}),\]
where the direct limit is taken over Hamiltonians $\tilde{H}$ satisfying
$\operatorname{slope}(\tilde{H})\notin \mathcal{S}(\alpha)$.

Finally, for any $-\infty \le a<b\le \infty$, we define the \textit{(filtered)
symplectic homology of $W$} by
\[SH^{(a,b)}_*(W):=\varinjlim_{H} HF^{(a,b)}_*(H),\]
where the direct limit is taken over Hamiltonians $H$ that are linear at
infinity and satisfy $H|_{W}<0$. In particular, since admissible Hamiltonians
form a cofinal family among these, we can take the direct limit over that class of Hamiltonians.

 With this definition, for any admissible Hamiltonian $H$ with slope $\text{slope}(H)=a$ and any $t<a$, there is a natural map 
\begin{equation} \label{eq: inclusion maps}
    i_H^t: HF^{t}_*(H) \rightarrow SH^t_*(W).
\end{equation}

\begin{remark}[Cofinal sequences of Hamiltonians]\label{rmk: cofinal sequence/families}
As mentioned at the beginning of this section, if $H$ is admissible, then
$H-H|_{W}$ is semi-admissible. Choose sequences $s_i\to\infty$ and
$\epsilon_i\to 0$. For any sequence $H_{s_i}$ of semi-admissible Hamiltonians
with $\operatorname{slope}(H_{s_i})=s_i$, the sequence $H_{s_i}-\epsilon_i$ is cofinal, and the filtered homologies of $H_{s_i}$ and $H_{s_i}-\epsilon_i$
agree up to a shift by $\epsilon_i$ of the action window. In particular, for
any $-\infty<a<b<\infty$,
\[SH^{(a,b)}_*(W)=\varinjlim_{i\to\infty} HF^{(a,b)}_*(H_{s_i}).\]
In other words, we may use semi-admissible Hamiltonians to compute filtered symplectic homology in bounded action intervals. Another convenient cofinal family is obtained by fixing a (semi-)admissible
Hamiltonian $H$ and considering the multiples $sH$ for $s>0$. In this case,
\[SH^{(a,b)}_*(W)=\varinjlim_{s\to\infty} HF^{(a,b)}_*(sH).\]
See \cite{cineli2025barcode, cineli2026invariant, fernandes2024barcode} for
detailed proofs.
\end{remark}

Note that, similarly to Floer homology, there are ``inclusion'' maps for
symplectic homology: for $a<a'$ and $b<b'$,
$$SH^{(a,b)}_*(W)\to SH^{(a',b')}_*(W),$$
defined by taking the direct limit of the corresponding ``inclusion'' maps in
Floer homology. Furthermore, symplectic homology fits into a long exact
sequence analogous to \eqref{eq: long exa seq for Hamiltonian}:
\begin{equation}\label{eq: long exa seq for symplectic hom}
\cdots \to SH_*^{(a,b)}(W)\to SH_*^{(a,c)}(W)\to SH_*^{(b,c)}(W)\to SH_{*-1}^{(a,b)}(W)\to \cdots
\end{equation}
Symplectic homology admits a decomposition with respect to
$H_1(W;\Z)$, inherited from \eqref{eq: H_1 splitting of HF}, namely
\begin{equation}\label{eq: H_1 splitting in SH}
SH^{(a,b)}_*(W)=\bigoplus_{\eta\in H_1(W;\Z)} SH^{(a,b)}_*(W;\eta),
\end{equation}
and this decomposition is preserved by the long exact sequence
\eqref{eq: long exa seq for symplectic hom}.

\begin{remark}
One can further refine the decomposition of symplectic homology in \eqref{eq: H_1 splitting in SH} by free homotopy classes of loops in $W$, which corresponds to the connected components of the free loop space of $M$ when $\widehat{W}=T^*M$. We denote the symplectic homology associated to a free homotopy class of loops $\mathfrak{f}$ by $SH_*^{(a,b)}(W,\mathfrak{f})$. If $\mathfrak{f}$ projects to $\eta\in H_1(M;\Z)$, then $SH_*^{(a,b)}(W,\mathfrak{f}) \neq 0$ implies $SH_*^{(a,b)}(W;\eta)\neq 0$. This is used in Theorem \ref{thm: main_result}.
\end{remark}

We note that there is a useful relationship between symplectic homology and the Floer homology of a semi-admissible Hamiltonian, which we describe below.

\begin{theorem}[\cite{cineli2025barcode}, Theorem~3.5]\label{thm: Ham to symp}
Let $H$ be a semi-admissible Hamiltonian with $\operatorname{slope}(H)=a$.
Then, for every $\tau<a$, there exists an isomorphism
\[\Phi_H^\tau \colon SH^{\tau}_*(W,\lambda) \xrightarrow{\cong} HF^{a_h(\tau)}_*(H),\]
where $a_h$ is defined in \eqref{eq: a_h function}.

Moreover, these isomorphisms are natural in the sense that they commute with
action ``inclusion'' maps and monotone continuation maps. More precisely, for
any $\tau'\le \tau\le a$ and any two semi-admissible Hamiltonians $H'\le H$,
the following diagrams commute:
\begin{equation}\label{diag: diagram sympl to Hamil}
\begin{tikzcd}
SH^{\tau'}_*(W,\lambda) \arrow{r}{\Phi_{H}^{\tau'}} \arrow{d}
  & HF^{a_h(\tau')}_*(H) \arrow{d} \\
SH^{\tau}_*(W,\lambda) \arrow{r}{\Phi_H^{\tau}}
  & HF^{a_h(\tau)}_*(H),
\end{tikzcd}
\end{equation}
where the vertical arrows are the inclusion maps, and
\begin{center}
\begin{tikzcd}
SH^{\tau}_*(W,\lambda) \arrow{r}{\Phi_{H'}^{\tau}} \arrow{d}{id}
  & HF^{a_h(\tau)}_*(H') \arrow{d} \\
SH^{\tau}_*(W,\lambda) \arrow{r}{\Phi_H^{\tau}}
  & HF^{a_h(\tau)}_*(H),
\end{tikzcd}
\end{center}
where the right vertical arrow is the continuation map from $H'$ to $H$.
\end{theorem}

It follows from the first diagram \eqref{diag: diagram sympl to Hamil}, the
exact sequences \eqref{eq: long exa seq for symplectic hom} and
\eqref{eq: long exa seq for Hamiltonian}, together with the five lemma, that
for any $(a,b)$ with $0<a<b<\operatorname{slope}(H)$ there is an isomorphism
\begin{equation}\label{eq: iso filt SH to HF}
\Phi_H^{(a,b)} \colon SH^{(a,b)}_*(W,\lambda) \xrightarrow{\cong}
HF^{(a_h(a),a_h(b))}_*(H).
\end{equation}

Hamiltonian Floer homology and symplectic homology both carry a
\textit{pair-of-pants product}. In the Hamiltonian case, it takes the form
$$HF^{a}_{m_0}(H_0)\otimes HF^{b}_{m_1}(H_1)\to HF^{a+b}_{m_0+m_1-n}(H_2),$$
provided that $\operatorname{slope}(H_i)\notin \mathcal{S}(\alpha)$ for
$i=0,1,2$ and that
$\operatorname{slope}(H_0)+\operatorname{slope}(H_1)\le \operatorname{slope}(H_2)$.

Such a product is defined at the chain level, for non-degenerate Hamiltonians
linear at infinity, by counting solutions to a version of Floer's equation
\eqref{eq: floer eq} on a three-punctured Riemann sphere equipped with
cylindrical ends. It is then extended to (semi-)admissible Hamiltonians via
continuation maps. The product is natural with respect to continuation maps and
induces the \textit{pair-of-pants product on symplectic homology} by taking a
direct limit of the following diagrams:
\begin{equation}\label{eq: pair of pants}
\begin{tikzcd}[column sep=5em, row sep=3.5em]
HF_{m_0}^a(H_0)\otimes HF_{m_1}^b(H_1) \arrow[r] \arrow[d, "i_{H_0}\otimes i_{H_1}"'] &
HF^{a+b}_{m_0+m_1-n}(H_2) \arrow[d, "i_{H_2}"] \\
SH_{m_0}^{a}(W)\otimes SH_{m_1}^{b}(W) \arrow[r] &
SH_{m_0+m_1-n}^{a+b}(W),
\end{tikzcd}
\end{equation}
where the vertical maps are the inclusion maps from \eqref{eq: inclusion maps}.
See \cite{abouzaid2013symplectic} for details of the construction.

The \textit{pair-of-pants product in symplectic homology} takes the form
\[SH_{m_1}(W)\otimes SH_{m_2}(W)\to SH_{m_1+m_2-n}(W).\]

\subsection{Transfer map}
In this subsection, we recall the main ideas behind the \textit{transfer map},
first defined by Viterbo in \cite{viterbo1999functors} and subsequently studied
by many others. Our setup is closest to \cite{cieliebak2018symplectic}.

Let $(W_\pm,\lambda_\pm)$ be Liouville domains of the same dimension, and let
$i\colon W_-\to W_+$ be an exact symplectic embedding with
$i(W_-)\subset \operatorname{int}(W_+)$. We will identify $W_-$ with its image
and hence treat $i$ as an inclusion $W_-\subset W_+$. The \textit{transfer map}
is a map between filtered symplectic homologies
\begin{equation}\label{eq: tranfer map}
f_{W_-W_+}^{(a,b)} \colon SH_*^{(a,b)}(W_+) \to SH_*^{(a,b)}(W_-),
\end{equation}
satisfying a functoriality property. More precisely, for a triple exact inclusion of
Liouville domains $W_-\subset W\subset W_+$, one has
\begin{equation}\label{eq: functoriality tranfer map}
f_{W_-W_+}^{(a,b)} = f_{W_-W}^{(a,b)} \circ f_{WW_+}^{(a,b)}.
\end{equation}

Let us recall the definition of $f_{W_-W_+}^{(a,b)}$. Denote by
$\mathcal{H}_{<0}(W_+;W_-)$ the set of Hamiltonians
$H\colon \widehat{W}_+\times S^1 \to \R$ that are linear at infinity and satisfy
$H|_{W_-\times S^1}<0$. Although such Hamiltonians are not linear
at infinity for $\widehat{W}_-$, they nevertheless form a suitable class for
computing $SH_*(W_-)$.

\begin{lemma}\cite[Lemma~5.1]{cieliebak2018symplectic}\label{lemma: transfer map lemma}
For any $-\infty<a<b<\infty$, one has
$$SH^{(a,b)}_*(W_-)=\varinjlim_{H\in \mathcal{H}_{<0}(W_+;W_-)} HF_*^{(a,b)}(H).$$
\end{lemma}

We remark that our setup is simpler than the one in \cite{cieliebak2018symplectic},
since we only consider the special class of Liouville cobordisms that are actual
Liouville domains.

Define $\mathcal{H}_{<0}(W_\pm)$ to be the set of non-degenerate Hamiltonians
$H\colon \widehat{W}_\pm\times S^1\to \R$ that are linear at infinity and satisfy
$H|_{W_\pm\times S^1}<0$. Note that $\mathcal{H}_{<0}(W_+)\subset
\mathcal{H}_{<0}(W_+;W_-)$, and for $K\in \mathcal{H}_{<0}(W_+)$ and
$H\in \mathcal{H}_{<0}(W_+;W_-)$,$K\leq H$, there is a continuation map
\begin{equation}\label{eq: continutation tranfer}
f_{HK}^{(a,b)}\colon HF^{(a,b)}_*(K)\to HF^{(a,b)}_*(H).
\end{equation}
These maps define a morphism of directed systems. In light of
Lemma \ref{lemma: transfer map lemma}, taking the direct limit yields the
transfer map \eqref{eq: tranfer map}:
\begin{equation}\label{eq: transfer map as direct limit}
f_{W_-W_+}^{(a,b)} :=
\varinjlim_{\substack{K\le H\\ K\in\mathcal{H}_{<0}(W_+)\\ H\in\mathcal{H}_{<0}(W_+;W_-)}}
f_{HK}^{(a,b)}.
\end{equation}
Functoriality \eqref{eq: functoriality tranfer map} then follows from the
corresponding property for continuation maps; see
\cite[Proposition~5.4]{cieliebak2018symplectic} for a proof.

\subsection{Cotangent bundles and Viterbo's theorem}\label{subsec: SH of a cotangent bdle}
The goal of this subsection is to present a version of Viterbo's isomorphism
for fiberwise star-shaped Liouville subdomains in the cotangent bundle of a
closed smooth manifold $M$. Throughout this subsection, all homology groups are
taken with coefficients in $\Z_2$ unless explicitly stated otherwise.

Consider the free loop space $\Lambda M$. Recall that its homology, together
with the Chas-Sullivan product, forms an associative and commutative algebra;
see \cite{chas1999string,latschev2015free}.

The cotangent bundle $T^*M$ is an exact symplectic manifold with symplectic
form $\omega=d\lambda_{std}$, where the Liouville form $\lambda_{std}$ is
expressed in local coordinates $\{(p_i,q_i)\}$ as
$$\lambda_{std}=\sum_i p_i\,dq_i.$$
The corresponding Liouville vector field, i.e., the vector field $Z$
satisfying $i_{Z}\omega=\lambda_{std}$, is expressed in local coordinates as
\begin{equation}\label{eq: Livouvile vf}
Z=\sum_i p_i\,\partial_{p_i},
\end{equation}
which is the radial vector field.

For a Riemannian metric $g$ on $M$, the \textit{unit disk cotangent bundle}
$D_g^*M$, consisting of covectors of length at most one, is a Liouville
subdomain of $(T^*M,\lambda_{std})$. Viterbo's theorem provides an isomorphism
of filtered algebras
\begin{equation}\label{eq: Viterbo isomo}
SH_*(D_g^*M)\xrightarrow{\cong} H_*(\Lambda M),
\end{equation}
intertwining the pair-of-pants and the Chas-Sullivan products, respectively;
see \cite{abbondandolo2010floer, abouzaid2013symplectic, viterbo1999functors}.

For the remainder of this subsection, we explain how the algebra isomorphism
\eqref{eq: Viterbo isomo} can be generalized to fiberwise star-shaped
subdomains of $T^*M$ (though it is not action preserving). Although this
is well known to experts, we include the argument for completeness; see \cite[Section 7]{seidel2008biased}.

The completions of two Liouville domains $(W_1,\lambda_1)$ and $(W_2,\lambda_2)$
are said to be \textit{exact symplectomorphic} if there exists a diffeomorphism
$\phi\colon \widehat{W}_1\to \widehat{W}_2$ such that $\phi^*\lambda_2-\lambda_1$
is a compactly supported exact $1$-form. For any two Liouville domains whose
completions are exact symplectomorphic, there is a natural isomorphism of
algebras
\begin{equation}\label{eq: iso in algebra}
SH_*(W_1,\lambda_1)\cong SH_*(W_2,\lambda_2).
\end{equation}

This isomorphism can be described as follows. Let $H$ be any Hamiltonian on
$\widehat{W}_1$ that is linear at infinity, and let $J$ be any regular almost
complex structure for $H$. Pulling back by $\phi$ yields an induced Hamiltonian
$H'$ and an almost complex structure $J'$ on $\widehat{W}_2$, which are again
linear at infinity (with the same slope) and admissible, respectively. The
corresponding chain complexes are then naturally isomorphic, and hence
$HF(H,J)\cong HF(H',J')$. Moreover, this isomorphism preserves the pair-of-pants
product \eqref{eq: pair of pants}. Taking a direct limit yields the induced
isomorphism \eqref{eq: iso in algebra}; see \cite{mclean2008symplectic} for
more details.

We observe that the completion of any fiberwise star-shaped subdomain
$W\subset T^*M$ is exact symplectomorphic to $T^*M$. Indeed, let
$\Sigma=\partial W$. Since the Liouville vector field \eqref{eq: Livouvile vf}
is transverse to $\Sigma$, the restriction $\alpha=\lambda_{std}|_{\Sigma}$ is
a contact form on $\Sigma$, and hence $(W,\lambda_{std}|_{W})$ is a Liouville
domain. The flow of $Z$ induces a diffeomorphism
$T^*M\setminus \operatorname{int}(W)\to \Sigma\times [1,\infty)$ under which
$\lambda_{std}$ corresponds to $r\alpha$, where $r\in[1,\infty)$. This yields an
exact symplectomorphism between $\widehat{W}$ and $T^*M$.

Therefore, by choosing a Riemannian metric $g$ on $M$, the unit disk cotangent bundle
$D_g^*M$ is a fiberwise star-shaped subdomain. Combining the algebra
isomorphism $SH_*(W)\cong SH_*(D_g^*M)$ with Viterbo's isomorphism
\eqref{eq: Viterbo isomo}, we obtain an algebra isomorphism
\begin{equation}\label{eq: iso vit for star-shaped}
SH_*(W)\xrightarrow{\cong} H_*(\Lambda M).
\end{equation}
We remark that the isomorphism \eqref{eq: iso vit for star-shaped} does not
preserve action filtration, although it does preserve the
$H_1(M;\Z)$-splitting \eqref{eq: H_1 splitting in SH}.

\section{Local Floer homology}
Most of what we discuss in this section can be formulated in greater
generality, but we choose to work with a Liouville domain
$(W,\omega=d\lambda)$ satisfying $c_1(W)=0$ for convenience. We refer to
\cite{ginzburg2010conley, ginzburg2010local, fender2020local, mclean2012local}
and the references therein for more details. 

\subsection{Local Hamiltonian Floer homology}
Let $x$ be an isolated 1-periodic orbit of a Hamiltonian
$H\colon \widehat{W}^{2n}\times S^1\to \R$. Pick a sufficiently small
neighborhood $U$ of $x$. More precisely, we take $U$ to be a neighborhood of
the loop $x(S^1)$, viewing $x$ as its graph embedding in the extended space
$\widehat{W}^{2n}\times S^1$. Now let $\tilde{H}$ be a Hamiltonian that is $C^2$-close to $H$, agrees with $H$
outside of $U$, and such that all $1$-periodic orbits of $\tilde{H}$ contained
in $U$ are non-degenerate; see \cite[Theorem~9.1]{salamon1992morse}. We define the \textit{local Floer homology}
of $H$ at $x$, denoted by $HF_*(H,x)$ (and, when necessary, by $HF_*(H,x,\omega)$
to indicate the dependence on the underlying symplectic form), as the homology
of the complex generated by the $1$-periodic orbits of $\tilde{H}$ in $U$, with
differential given by counting solutions of Floer's equation for $\tilde{H}$
connecting these orbits. It is standard that this construction is well defined
(the relevant Floer cylinders and broken trajectories for $\tilde{H}$ remain in
$U$) and that the resulting group is independent of the choices (independent
of the perturbation $\tilde{H}$ and the almost complex structure). See
\cite{floer1995transversality, salamon1992morse} for more details. In fact, we can replace an isolated orbit by an isolated family of one-periodic orbits $\mathcal{F}$ of $H$, so that its local Floer homology is still defined in the same way and denoted by $HF_*(H,\mathcal{F})$; see \cite[Section 2.3]{mclean2012local}, \cite{fender2020local} and references therein. Note that, from the definition, the Hamiltonian $H$ need not be defined on all of
$\widehat{W}$. Since the construction is local, it suffices to consider
Hamiltonians defined on a neighborhood of $x$.

Recall that for two Hamiltonians $G$ and $H$, their composition is defined by
\begin{equation}\label{eq:composition Hamiltonians}
(H\#G)_t = H_t + G_t\circ (\varphi_H^t)^{-1},
\end{equation}
and its Hamiltonian flow is $\varphi_H^t\circ \varphi_G^t$. Suppose that
$\varphi_G^t$ is a local loop of Hamiltonian diffeomorphisms, defined on a
neighborhood $U$ of $x$ and generated by a Hamiltonian $G$. Then one has
$$  HF_*(G\# H,\varphi_G(x)) = HF_{*-2\mu}(H,x),
$$
where $\mu$ is the Maslov index of the loop $t\mapsto \varphi_G^t(x(0))$ and $\varphi_G(x)$ is the orbit $\varphi_G^t(x(t))$.

This property allows us to define the local Floer homology of a local Hamiltonian
diffeomorphism $\varphi$ with an isolated fixed point $p$, up to a shift in
degree, as follows. Choose a local Hamiltonian $H$ with $\varphi_H^1=\varphi$
and set $x(t)=\varphi_H^t(p)$. We then define $HF_*(\varphi,p):=HF_*(H,x)$.
The grading depends on the class of the path $\varphi_H^t$ in the universal
cover of the group of germs of Hamiltonian diffeomorphisms.

The next proposition shows how local Floer homology groups are building blocks for the ordinary Floer homology.
\begin{proposition} \label{prop: ham floer to local}
If $c\in \mathcal{S}(H)$ is isolated, then there exists $\epsilon>0$ 
such that
\begin{equation} \label{eq: splitning in SH}
    HF^{I}_*(H) = \bigoplus_{\mathcal{A}_H(\gamma) = c} HF_*(H,x),
\end{equation}
for any interval $I$ containing $c$ with length $|I|<\epsilon$.
\end{proposition}
The proof follows from standard arguments in local Floer homology and is omitted. See \cite{ginzburg2010conley} for details.

The \textit{support} of $HF_*(H,x)$ is, by definition, the set of integers $k$
such that $HF_k(H,x)\neq 0$, and we denote it by $\text{supp}(HF_*(H,x))$.
It is a consequence of \eqref{eq: mean index and conley zehnder index} that for an isolated orbit $x$,
\begin{equation}\label{eq: supp of local floer}
\text{supp}(HF(H,x))\subset [\Delta_H(x)-n,\Delta_H(x)+n].
\end{equation}

We now recall the definition of a \textit{symplectically degenerate maximum},
or simply an SDM. Let $\varphi$ be a Hamiltonian diffeomorphism and $p$ a fixed point. Consider $\varphi_H = \varphi$ a lift of $\varphi$ to the universal cover of the group of germs of Hamiltonian diffeomorphisms, and $x(t) = \varphi_H^t(p)$ the corresponding one-periodic orbit.

\begin{definition} \label{def: sdm ham}
An isolated fixed point $p$ of a local Hamiltonian diffeomorphism $\varphi$ is called a \textit{symplectically degenerate maximum} if $HF_{\Delta_H(x)+n}(H,x) \neq 0$ for some lift $\varphi_H = \varphi$.
\end{definition}

Note that the definition of an SDM does not depend on the trivialization along $x$. Moreover, for any lift $\varphi_H = \varphi$, the corresponding orbit $x$ is totally degenerate, i.e., all eigenvalues of the linearized return map are equal to one, and hence $\Delta_H(x) \in 2\Z$. In particular, one can choose a lift so that the corresponding orbit has mean index $0$;
See \cite{ginzburg2010conley,ginzburg2010local,ginzburg2009action} for properties and equivalent definitions of SDMs.
The geometric characterization of SDMs is what allows us to carry out local computations.
\begin{proposition}[\cite{ginzburg2010conley}, Proposition 4.5] \label{prop: geome characterization of SDM}
A fixed point $p = x(0)$ of $\varphi_H$  with $\Delta_H(x)=0$ is an SDM if and only if there exists a sequence of loops of Hamiltonian diffeomorphism loops $\eta_i$ with
$\eta_i^t(p) = x(t)$ such that the Hamiltonians $K_i$ defined near $p$ by
$$\varphi_H = \eta^t_i \circ \varphi_{K_i},$$
satisfy the following:
\begin{enumerate}
    \item[(K1)] The point $p$ is a strict local maximum of $K^i_t$ for all $t \in S^1$ and all $i$;
    \item[(K2)] There exists a symplectic basis $\Xi^i$ of $T_pM$ such that 
    $$||d^2(K_t^i)||_{\Xi^i} \rightarrow 0$$ uniformly in $t$; see Remark \ref{rmk: symplectic basis};
    \item[(K3)] The linearization of the loop $\eta^{-1}_i \circ \eta_j$ at $p$ is the identity map for all $i$ and $j$, i.e., $d((\eta_i^t)^{-1}\circ \eta_j^t)_p = Id$ for all $t \in S^1$ and, moreover, the loop $(\eta^t_i)^{-1} \circ \eta_j^t$ is contractible to $id$ in the class of loops fixing $p$ and having identity linearization at $p$. 
    \end{enumerate}
\end{proposition}

\begin{remark} \label{rmk: iterate of a SDM is a SDM}
    We note that if $p$ an SDM of $\varphi$ and is isolated for all iterations $\varphi^k$, then $p$ is also an SDM of $\varphi^k$ for all $k \in \mathbb{N}$. This follows from the SDM characterization in the paragraph following \cite[Example~5.7]{ginzburg2010local}.
\end{remark}
\begin{remark} \label{rmk: symplectic basis}
        For a basis $\Xi$ of a vector space $V$, and a vector $v \in V$, the norm $||v||_{\Xi}$ is the norm of $v$ with respect to the inner product for which $\Xi$ is orthonormal. The norm of an operator $V\rightarrow V$ is defined similarly.
\end{remark}
 
\subsection{Local symplectic homology} \label{sec: local SH}
Let $(W,\lambda)$ be a Liouville domain of dimension $2n$, $(N,\alpha)$ its contact boundary, and
let $\gamma$ be an isolated (not necessarily simple) Reeb orbit of period $T$.
Let $L$ be a (semi-)admissible Hamiltonian on $\widehat{W}$ with
$T<\operatorname{slope}(L)$. The associated $1$-periodic orbit
$\tilde{\gamma}$ of $L$ is also isolated. This Hamiltonian orbit determines an
isolated family $\Gamma=\tilde{\gamma}(S^1)$ of fixed points of the Hamiltonian
diffeomorphism $\varphi_L$. We define the local symplectic homology of $\gamma$,
denoted by $SH_*(\gamma)$, to be $HF_*(L,\tilde{\gamma})$, equivalently the
local Floer homology of the isolated family of fixed points $\Gamma$.

More concretely, let $U\subset N$ be a tubular neighborhood of $\gamma$
containing no other Reeb orbits with period close to $T$. Consider a
neighborhood $U\times (T-\delta,T+\delta)$ in the symplectization of $N$, with
coordinates $(x,r)$, and a Hamiltonian $L=l(r)$ such that $l'(r_0)=T$ for some
$r_0$ and $l''(r)>0$. Then $SH_*(\gamma)=HF_*(L,\tilde{\gamma})$, where
$\tilde{\gamma}$ is the corresponding Hamiltonian orbit.

As in the case of Floer homology in Proposition \ref{prop: ham floer to local}, on a small action window, symplectic homology may be expressed in terms of local symplectic homology as follows.
\begin{proposition} \label{corolary: symp hom to local}
If $c\in \mathcal{S}(\alpha)$ is isolated, then there exists $\epsilon>0$ 
such that
\begin{equation} \label{eq: splitning in SH}
     SH^{I}_*(W) = \bigoplus_{\mathcal{A}(\gamma) = c} SH_*(\gamma),
\end{equation}
for any interval $I$ containing $c$ with length $|I|<\epsilon$.
\end{proposition}
\begin{proof}
Consider  $\epsilon>0$ sufficiently small such that $(c-\epsilon,c+\epsilon) \cap \mathcal{S}(\alpha) = \{c\}$.
Let $L$ be a semi-admissible Hamiltonian with $\text{slope}(L)>c$ sufficiently large,
it follows from \eqref{eq: iso filt SH to HF} that
$$SH^I_*(W) = HF^{a_l(I)}_*(L).$$ 
By \eqref{eq: bilipchitz}, \eqref{eq:inequality action function} and \eqref{eq: A_h(t)convergetot} and the corresponding decomposition for local Floer homology (Proposition \ref{prop: ham floer to local}), the result follows.

\end{proof}

Similar to the Hamiltonian case, the support of $SH_*(\gamma)$, denoted by $\text{supp } (SH_*(\gamma))$, is the set of integers $k$ such that $SH_k(\gamma) \neq 0$. Analogously to \eqref{eq: mean index and conley zehnder index}, one has
\begin{equation} \label{eq:local symp hom support}
    \text{supp} (SH(\gamma))\subset [\Delta(\gamma) - n+1, \Delta(\gamma) + n];
\end{equation}
see \cite[Section~2.4]{cineli2024closed}. The following lemma provides a local model for a neighborhood of a simple Reeb orbit.

\begin{lemma}[\cite{hryniewicz2015local}, Lemma 5.2] \label{Lemma: local model reeb orbits} Let $\gamma$ be a simple Reeb orbit with action $\mathcal{A}(\gamma)=T$. Then there exists a tubular neighborhood $U \cong B\times S^1$, where $B \subset \R^{2n-2}$ is an open ball centered at the origin with coordinates $x = (q,p)$, and a Hamiltonian $H : U \rightarrow \R$ such that $\alpha = \lambda_0 + Hdt$, with $H_t(0) = T$,  $dH_t(0) = 0$ for all $t$, and $\lambda_0 = \frac{1}{2}(qdp - pdq)$. In particular, the image of $\gamma$ corresponds to $\{0\}\times S^1$.
    
\end{lemma}
The proof of Lemma \ref{Lemma: local model reeb orbits} is standard; see \cite{hryniewicz2015local}. We touch upon a similar construction in the first step of the proof of Theorem \ref{thm: SDM gene orbits}.

Note that in such a neighborhood, the Reeb vector field is given by 
\begin{equation} \label{eq: Reeb vf in nbh}
    R_{\alpha} = \frac{1}{H - \lambda_0(X_H)}(\partial_t -X_H).
\end{equation}
In light of Lemma \ref{Lemma: local model reeb orbits}, and \eqref{eq: Reeb vf in nbh}, for a simple Reeb orbit $\gamma$,  the return map $\varphi$ in the transversal section $B\times \{1\}$ coincides with $\varphi = \varphi_{-H}$. The relationship between the local symplectic homology and the local Floer homology of the return map is given by the following Theorem: 

\begin{theorem}[\cite{fender2020local}, Theorem 1.1] \label{thm: Fender} Let $\gamma$ be a simple isolated Reeb orbit with $\gamma(0)=p$. Then we have the following isomorphisms in local homology for any coefficient ring $R$,
\begin{equation}
\label{eq: symplecti hom to Ham homo}
    SH_*(\gamma) \cong HF_{*}(\varphi,p) \oplus HF_{*-1}(\varphi,p).
\end{equation}

\end{theorem}
Once a local model as in Lemma \ref{Lemma: local model reeb orbits} has been chosen, there is a canonical trivialization along the orbit so the grading is well defined.

Similarly to Hamiltonian case \ref{def: sdm ham}, we have the following definition; cf. \cite{ginzburg2013closed,hryniewicz2015local}. 

\begin{definition} \label{def: SDM}
We say that an isolated Reeb orbit $\gamma$ is a Reeb SDM if $p = \gamma(0)$ is an SDM for the return map $\varphi$, i.e., $HF_{\Delta(\gamma) + n-1}(\varphi,p) \neq 0$. 
\end{definition}
Note that the notion of a Reeb SDM is independent of the trivialization of $\xi$ along $\gamma$. Moreover, a Reeb SDM is totally degenerate, hence $\Delta(\gamma) \in 2\Z$. Hence, by choosing a suitable trivialization, one can assume that $\Delta(\gamma)= 0$; see Remark \ref{remark: mean index coincide}.
\begin{remark}
Note that the notion of Reeb SDM is defined only when $\text{dim }N \geq 3$, or equivalently $\text{dim }W\geq 4$, since otherwise a Reeb orbit has no non-trivial return map. This explains why in Theorems \ref{thm: main_result} and \ref{thm: SDM gene orbits} we require $\text{dim } W \geq 4$.
\end{remark}

\begin{remark}
\label{remark: mean index coincide}
     In a local model of simple Reeb orbit $\gamma$ as in Lemma \ref{Lemma: local model reeb orbits}, for any Hamiltonian $L$ in the symplectization with $\text{slope}(L)>T$, the Reeb orbit $\gamma$, the corresponding Hamiltonian orbit $\Tilde{\gamma}$ and the fix point $p =\gamma(0)$ of the return map all have the same the mean index when computed in their respective settings using the induced trivializations; see \cite{fender2020local}, Corollary 1.3. We may denote it by $\Delta(\gamma)$, $\Delta_L(\tilde{\gamma}) $, or $ \Delta_H(p)$, depending on the context.
\end{remark}
Note that, from Theorem \ref{thm: Fender} and Remark \ref{remark: mean index coincide}, a simple Reeb orbit $\gamma$ is a Reeb SDM if and only if $SH_{\Delta(\gamma)+n}(\gamma) \neq 0$.
We point out that in the proof of Theorem \ref{thm: Fender} in \cite{fender2020local}, it is essential that the Reeb orbit to be simple. It is conjectured that for a simple Reeb orbit $\gamma$ such that all the iterates are isolated, it holds that
$$SH_*(\gamma^k) = HF_*(\varphi^k,p)^{\Z_k} \oplus HF_{*-1}(\varphi^k,p)^{\Z_k},$$
for all $k \in \mathbb{N}$, at least when the homology groups above are taken with coefficients in $\mathbb{Q}$; see \cite{ccineliclosed}. Nonetheless, one always has that 
\begin{equation} \label{eq: symp hom bound by hom of the return map}
    \text{dim }SH_*(\gamma^k) \leq \text{dim } HF_*(\varphi^k) + \text{dim } HF_{*-1}(\varphi^k).
\end{equation}
See \cite[Lemma 3.4]{mclean2012local} for a proof, and \cite{ccineliclosed} for a proof using coefficients in $\mathbb{Q}$.
As consequence of Theorem \ref{thm: Fender}, Remark \ref{remark: mean index coincide}, \eqref{eq: supp of local floer}, and \eqref{eq: symp hom bound by hom of the return map}, the following corollary holds.
\begin{corollary} \label{cor: SDM recognition}
Let $\gamma$ be a not necessarily simple isolated Reeb orbit. If 
\begin{equation} \label{eq: Reeb orbit eq being SDM}
    SH_{\Delta(\gamma) +n}(\gamma) \neq 0,
\end{equation}
then $\gamma$ is a Reeb SDM.
\end{corollary}
\begin{proof}
Assume that $\gamma = z^k$, where $z$ is a simple Reeb orbit. Since $\gamma$ is isolated, it follows that $z$ is isolated. Let $p = z(0)$ and $\varphi$ be the return map of $z$ at $p$. Note that $\varphi^k$ is the return map of $\gamma$ at $p=\gamma(0)$. By \eqref{eq: symp hom bound by hom of the return map}, \eqref{eq: Reeb orbit eq being SDM}, and Remark \ref{remark: mean index coincide},  we conclude that $HF_{\Delta(p)+n-1}(\varphi^k,p)\neq 0$. Hence $\gamma$ is a Reeb SDM.
\end{proof}

\subsection{Spectral Invariants} \label{sec: spectral invariants}
Let $(W,\lambda)$ be a Liouville domain and $H$ a Hamiltonian on $\widehat{W}$ that is linear at infinity. Throughout this subsection, we assume that $\text{slope}(H) \notin \mathcal{S}(\alpha)$. For any $\beta  \in HF_*(H)$, we define its spectral invariant as 
\begin{equation} \label{eq: spc invari for H}
    c_{\beta}(H) = \inf_{a}\{\beta \in Im(HF_*^a(H) \rightarrow HF_*(H))\}.
\end{equation}
If $H$ is non-degenerate, then \eqref{eq: spc invari for H} is equivalent to
\begin{equation} \label{eq: spec invariants for non deg}
    c_\beta(H)= \inf_{\beta = [\sum_i x_i]}  \mathcal{A}_H(\sum_i x_i),
\end{equation}
where the infimum is taken over all cycles representing $\beta$, and the action of a Floer cycle $\sum_i x_i$ is given by $\mathcal{A}_H(\sum_i x_i) = \max_{i}\{\mathcal{A}_H(x_i)\}$.

Similarly to \eqref{eq: spc invari for H}, for a Liouville domain $W$, the \textit{spectral invariant} of $\beta \in SH_*(W)$ is defined as 
\begin{equation}
    c_{\beta}(W) = \inf_{a}\{\beta \in Im(SH_*^a(W) \rightarrow SH_*(W))\}.
\end{equation}

By Theorem \ref{thm: Ham to symp}, for $H$ semi-admissible with $\text{slope}(H) > \tau$, the map $\Phi_H: SH^{\tau}(W) \rightarrow HF^{a_{h}(\tau)}(H)$ is a natural isomorphism. For $\beta \in im(SH_*^{\tau}(H) \rightarrow SH_*(W))$, we can consider the spectral invariant 
$$c_{\beta}(H) :=c_{\Phi_H(\beta)}(H).$$ 
Again by Theorem \ref{thm: Ham to symp}, we have that
\begin{equation} \label{eq: spec in from HF to SH}
    c_{\beta}(H) = a_h(c_{\beta}(W)).
\end{equation}

The spectral invariants for both Liouville domains and Hamiltonians enjoy several properties, and we list the ones that will be important later on. Consider $H,K$ linear at infinity Hamiltonians. Then:

\begin{itemize} 
    \item[(SI1)] \textit{Spectrality:} For $\beta \in HF_*(H)$, there exists a $1$-periodic orbit $x$ for $H$, denoted an \textit{action selector}, such that
    \begin{equation} \label{eq: SI1}
        c_{\beta}(H) = \mathcal{A}_H(x);
    \end{equation}
    \item[(SI2)] \textit{Sublinearity:} Consider $L$ another linear at infinity Hamiltonian with $\text{slope}(L)>\text{slope}(H)+\text{slope}(K)$. If $\theta\in HF_*(H)$ and $\beta\in HF_*(K)$, then the spectral invariant of the their pair of pants product satisfies
    \begin{equation} \label{eq: SI2}
        c_{\theta\cdot\beta}(L) \leq c_{\theta}(H)+c_{\beta}(K);
    \end{equation} 
    \item[(SI3)] \textit{Continuity:} Assume that $H,K$ are linear at infinity with $\text{supp}(H -K)$ compact, (in particular $\text{slope}(H) = \text{slope}(K)$), and  $\theta \in HF_*(H),\ \beta\in HF_*(K)$, then
    \begin{equation} \label{eq: SI3} 
        |c_{\theta}(H) - c_{\beta}(K)| \leq ||H-K||_{L^\infty}
    \end{equation}
    
    \item[(SI4)]      For $\theta,\beta \in SH_*(W)$, then 
    \begin{equation}\label{eq: SI4}
        c_{\theta\cdot\beta}(W) \leq c_{\theta}(W) + c_{\beta}(W).
    \end{equation} \label{eq: sublinearity spectral invts}
    In particular, for an integer $m\geq 1$, then
    $$c_{\beta^m}(W)\leq m c_{\beta}(W). $$
\end{itemize}

Properties (SI1) and (SI3) are well known and we omit the proofs. The reader may check \cite{irie2014hofer, abbondandolo2006floer, abbondandolo2010floer} for more details. The setting in \cite{irie2014hofer} is probably the closest to ours. Properties (SI2) and (SI4) are less standard, so we present a proof for the sake of completeness. 
\begin{proof}[Proof of (SI2)] Let $\tilde{H},\tilde{K}$ and $\tilde L$ be $C^2$-small non-degenerate linear at infinity perturbations of $H$, $K$ and $L$, respectively, with $\text{supp}(H-\tilde{H})$, $\text{supp}(K-\tilde{K})$ and $\text{supp}(L-\tilde L)$ compact. 

By definition, we have that $HF_*^a(H)\cong HF_*^a(\tilde{H})$ and $HF_*^a(K)\cong HF_*^a(\tilde{K})$, for all $a$, and we still denote the images of the classes $\theta$ and $\beta$ under such isomorphism by the same symbol. For any chain representations, $\theta=[\sum x_i], \beta=[\sum y_j]$, where $x_i$ and $y_j$ are 1-orbits of $\tilde H$ and $\tilde K$, respectively, we denote by $x_i\cdot y_j \in CF(\tilde L)$ the chain representative of the pair-of-pants product of the two orbits $x_i$ and $y_j$. Then
$$A_{\tilde L}(\sum_{i,j} x_i\cdot y_j) \leq \max_{i,j}\{ A_{\tilde{H}}(x_i) + A_{\tilde{K}}(y_j)\} \leq c_{\theta}(\tilde{H}) + c_{\beta}(\tilde{K}).$$
From \eqref{eq: spec invariants for non deg}, we obtain
\begin{equation} \label{eq: spec invariants of the sum}
    c_{\theta\cdot \beta}(\tilde L) \leq c_{\theta}(\tilde{H}) + c_{\beta}(\tilde{K}).
\end{equation}
The inequality \eqref{eq: SI2} then follows from the continuity property \eqref{eq: SI3}.
\end{proof}

\begin{proof}[Proof of (SI4)] Let $\epsilon>0$ and $H$ a semi-admissible Hamiltonian $H$ with sufficiently large slope such that $a_{h}(t)\leq t + \epsilon/2$, for all $t < c_{\alpha}(W) + c_{\beta}(W)+1$; see \eqref{eq: A_h(t)convergetot}. Then from \eqref{eq: spec in from HF to SH} and \eqref{eq: SI2}, we have
\begin{align*}
    c_{\alpha\cdot \beta}(W) & \leq a_{2h}(c_{\alpha \cdot \beta}(W))\\
    & = c_{\alpha \cdot \beta}(2H) \\
    & \leq c_{\alpha}(H) + c_{\beta}(H) \\
    & =  a_h(c_{\alpha}(W)) + a_h(c_{\beta}(W)) \\
    & \leq c_{\alpha}(W) + c_{\alpha}(W) + \epsilon,
\end{align*} 
where the first inequality follows from \eqref{eq:inequality action function}, the second one from \eqref{eq: SI2} and the last one from the choice of Hamiltonian.  Since the above inequality holds for every $\epsilon>0$, we conclude \eqref{eq: SI4}.
\end{proof}
\begin{lemma} \label{Lemma: Reeb orbits with non zero loc hom}
    Let $\alpha \in SH_{k}(W;\eta) \setminus \{0\}$ and $\eta\in H_1(W;\Z)\setminus \{0\}$. Then there exists a Reeb orbit $\gamma$ representing the homology class $\eta$ with $\mathcal{A}(\gamma) =c_{\alpha}(W)$ and  
    \begin{equation} \label{eq: symp hom non zero lemma}
        SH_k(\gamma) \neq 0.
    \end{equation}
\end{lemma}
\begin{proof}
First, note that it follows from (SI1) and the properties  of the action function that $c_{\alpha}(W) = (a_h)^{-1}(\mathcal{A}_H(x))$, for $x$ a Hamiltonian orbit of a semi-admissible Hamiltonian $H$ representing the class $\eta$. Since $\eta$ is non-trivial, and $H|_W\equiv 0$, then $x$ corresponds to a Reeb orbit on $N$, so $c_{\alpha}(W)\in \mathcal{S}(\alpha)$.
    
Assume by contradiction that \eqref{eq: symp hom non zero lemma} does not hold for all such Reeb orbits. Then, by Corollary \ref{corolary: symp hom to local}  there exists $\epsilon>0$ sufficiently small so that
$$SH_k^{(c_{\alpha}(W) - \epsilon, c_{\alpha}(W)+\epsilon)}(W;\eta) = \bigoplus_{\substack{\mathcal{A}(\gamma) = c_{\alpha}(W)\\ [\gamma]=\eta}} SH_k(\gamma) = 0.$$
From the long exact sequence for symplectic homology \eqref{eq: long exa seq for symplectic hom}, we obtain
$$ SH_k^{c_{\alpha}(W) - \epsilon}(W;\eta) \rightarrow SH_k^{c_{\alpha}(W) + \epsilon}(W;\eta) \rightarrow SH_k^{(c_{\alpha}(W) - \epsilon,c_{\alpha}(W) + \epsilon)}(W;\eta)=0,$$
which implies the inclusion map $SH_k^{c_{\alpha}(W) -\epsilon}(W;\eta) \rightarrow SH_k^{c_{\alpha}(W)+\epsilon}(W;\eta)$ is surjective. This contradicts the definition of spectral invariant. 
    
\end{proof}
\section{Proof of the main theorems} \label{section: proof of main theorems}
In this section, we present the proofs of Theorems \ref{thm: SDM gene orbits}, \ref{thm: SDM generation on a covering}, and \ref{thm: main_result}, in that order. Although the proof of \ref{thm: SDM gene orbits} is long and technical, we present it first, since it is used in a crucial way in the proof of \ref{thm: main_result}. Readers who wish to avoid technicalities may assume Theorem \ref{thm: SDM gene orbits} and proceed directly to the proof of
Theorem \ref{thm: main_result}.
\subsection{Proof of Theorem \ref{thm: SDM gene orbits}}

In this subsection, we give a detailed proof of Theorem \ref{thm: SDM gene orbits}. The argument follows ideas from \cite{ginzburg2010conley} and \cite{ginzburg2013closed}. For the sake of readability, we divide the proof in a few step. Throughout, all the action intervals are assumed to have end points outside of the action spectrum of the contact forms involved.  

\begin{proof}

\textbf{Step I}. \textit{Local model for SDM:} Let $(N,\alpha)=(\partial W,\lambda|_{\partial W})$ and $\gamma$ be a simple SDM for the Reeb flow of $\alpha$ on $N$ whose action we denote by $c=\mathcal{A}(\gamma)$. With respect to the trivialization of the contact structure along $\gamma$ induced by a choice of trivialization of the
anti-canonical bundle of $W$, as in Subsection \ref{subsection: indexes}, the mean index satisfies
$\Delta=\Delta(\gamma) \in 2\Z$ because $\gamma$ is totally degenerate, and it need not be zero. We consider a trivialization $\tau$ of the contact structure along $\gamma$ with respect to which the mean index is zero. Such a trivialization exists since $\Delta \in 2\Z$. The trivialization $\tau$ along $\gamma$ induces trivializations $\tau^k$ along all iterates $\gamma^k$ by iteration, and also along any orbit, of any contact form supporting $\xi = \text{ker}(\alpha)$, homotopic to these iterates. Consequently, the indices of orbits homotopic to $\gamma^k$ computed with respect to these two trivializations differ by a multiple of $k\Delta$.

In what follows we construct a local model near $\gamma$ with $U = B\times S^1$ a tubular neighborhood, where $B = B_R(0)$ is a ball of small radius $R$ in $\mathbb{R}^{2n-2}$ with coordinates $x=(q,p)$, adapted to the trivialization $\tau$ in the obvious sense, and a $H : B\times S^1 \rightarrow (0,\infty)$ Hamiltonian such that
\begin{enumerate}
\item[(H1)] $\alpha = \lambda_0 + Hdt$, where $\lambda_0 = \frac{1}{2}(qdp - pdq)$ is the standard primitive of the standard symplectic form on $B$,

\item[(H2)] $H_t(0) = c$ is a strict maximum of $H_t$ for all $t \in S^1$.
\end{enumerate}

The construction of such a neighborhood follows from the local model for a simple Reeb orbit, together with Moser's argument. 

More precisely, since $\gamma$ is a simple Reeb orbit, by Lemma \ref{Lemma: local model reeb orbits}, there exists a neighborhood as above with (H2) replaced by $dH_t(0) = 0$ for all $t$. Given that $\gamma$ is a Reeb SDM, then $x(t) \equiv 0$ is an SDM for $H$ with $\Delta_H(x) = 0$. By Proposition \ref{prop: geome characterization of SDM}, there exists a loop of Hamiltonians diffeomorphisms $\eta^t = \eta^t_1$ whose linearization at the origin is the identity, i.e., $d(\eta^t)_0 = I$ for all $t$, such that $\varphi_H^t = \eta^t \circ \varphi_K^t$, where $K$ is a Hamiltonian having strict local maximum at the origin. Let $G$ be a Hamiltonian such that $\eta^t = \varphi_G^t$, which by normalization, we can assume $G_t(0) = 0$. The property $d(\eta^t)_0 = I$ for all $t$ is equivalent to $d^2(G_t)_0 = 0$ for all $t$. From the definition \eqref{eq:composition Hamiltonians} of the composition $H = G\#K$, we have
$$K_t =(H_t  -  G_t) \circ \varphi^t_G .$$
Consider the map $\psi : B\times S^1 \rightarrow B\times 
S^1$ defined by $\psi(x,t) = (\eta^t(x),t)$. Then
$$\psi^*(\alpha) = \lambda_0 + dF+ (K_t + \lambda_0(X_G)(\eta^t))dt,$$
for some function $F$. It is not hard to see that $K_t + \lambda_0(X_{G})(\eta^t)$ has the origin as a strict local maximum. By abuse of notation, we write 
$$\psi^*(\alpha) = \lambda_0 + dF + Hdt,$$
where $H$ has a strict local maximum at the origin. 

Consider the contact form $\alpha = \lambda_0 + Hdt$ and the homotopy $\alpha_s = (1-s)\psi^*(\alpha) + s \alpha$, where $s \in [0,1]$. Note that $\alpha_s$ are contact forms near $\{0\}\times S^1$. Let $Y_s = FR_{\alpha_s}$, where $R_{\alpha_s}$ is the Reeb vector field of $\alpha_s$, and let $\psi_s$ be the flow of $Y_s$. Since
$$\frac{d}{ds} \psi_s^*(\alpha_s) = 0, \text{ for all } s \in [0,1],$$
then $\psi\circ \psi_1$ gives the desired coordinates.

Since $dH_t(0) = 0$ for all $t\in S^1$ and $c>0$, by shrinking $B$ if necessary, we can further assume that
\begin{equation} \label{eq: Hamilt cond so graph is contact}
    H - \lambda_0(X_H)  \neq  0 \text{ in } B\times S^1.
\end{equation}

Now with this local model in hand, we consider functions $H_{+} : B \rightarrow (0,\infty)$ and $H_-:B\times S^1 \rightarrow \R$, $C^0$-close to $H$ such that
$$H_+ \geq H_t \geq H_-, \text{ for all } t \in S^1,$$
where equality holds only at $0 \in B$, and $H_{\pm }\equiv c_{\pm}>0$ on the region $Y = B_R \setminus B_{R/2}$ near the boundary of $B$, with $c_{\pm}$ constants. In particular, $c_+\geq c_-$. Denote by $\rho = ||z||^2/2$ the square norm on $B$ with respect to the chosen trivialization. Then we require from $H_+$ that 

\begin{itemize}
    \item[($\text{H}_+$)] $H_+$ is a decreasing function of $\rho$, constant equal to $c$ round $0$, and equal to $c_+>0$ for $\rho \in [R/2,R]$.
\end{itemize}
The function $H_-$ is taken such that $H_- = G \# F$, where $F$ is a bump function (see \cite[Section 7]{ginzburg2010conley}) on $B$ with respect to a symplectic basis, say $\Xi$ (see Remark \ref{rmk: symplectic basis}), and $G$ satisfies that $d^2G_t(0) = 0$ and $dG_t(0) = 0$ for all $t$, along with $G(0)=0$, and $G\equiv 0$ outside $Y$. Moreover 
\begin{itemize}
    \item[($\text{H}_-1$)] $F$ is a decreasing function of $\rho_{\Xi} = ||z||^2_{\Xi}/2$ with slope $|F'|<\pi$, equal to $c$ at $0$ and equal to $c_->0$ on $Y$, 
    \item[($\text{H}_-2$)] $F$ has a non-degenerate maximum at $0$, and the norm $||d^2F_0||_{\Xi}$ on the symplectic vector space $T_0B$ is sufficiently small; see Remark \ref{rmk: symplectic basis},
    \item[($\text{H}_-3$)] There exists a family of Hamiltonians $G^s$, $s\in [0,1]$ with $G^0 = G$ and $G^1\equiv 0$ satisfying the same properties as $G$ for all $s$, such that  $H_s = G^s \# F \leq H_+$. In particular $H_s \equiv c_-$ on $Y$ for all $s$, and $H_0 = H_-$.
\end{itemize}

We also assume that $H_s$ are all $C^0$-close to $H$. Moreover, both $H_+$ and $H_s$ for all $s$ satisfy condition \eqref{eq: Hamilt cond so graph is contact}. 

There are several other conditions that these bump functions must satisfy (see \cite[Section 7]{ginzburg2010conley}), but they won't be explicitly used here. The construction of $H_{\pm}$ depends on the choice of several parameters, which we will specify later. For now, we only assume that they satisfy the conditions above. Strictly speaking, we replace the bump functions described above by $C^2$-small perturbations (as in \cite{ginzburg2010conley}) and, by abuse of notation, continue to denote them by $H_{+}$ and $F$. 

\textbf{Step 2.} \textit{Symplectic embeddings:} Consider the contact manifold $(U = B\times S^1,\alpha = \lambda_0 + Hdt)$, and its symplectization $(U \times (0,\infty),d(r\alpha))$,  where $r$ is the $(0,\infty)$-coordinate. We can also consider the exact symplectic manifold $(U\times (0,\infty),d\tilde \alpha)$ with $\tilde\alpha=\lambda_0+hdt$, where $h$ denotes the $(0,\infty)$-coordinate. Note that condition \eqref{eq: Hamilt cond so graph is contact} ensures that $\Gamma_H$ is of contact type in $(U\times (0,\infty), d\tilde{\alpha})$. It is not hard to see that there exists a strict contactomorphism between $(U,\alpha)$ and the graph of $H$ on $(U\times (0,\infty),d\tilde{\alpha})$, i.e., $(\Gamma_H,\tilde{\alpha}|_{\Gamma_H})$. 

Moreover, we can identify a neighborhood of the graph $\Gamma_H $ in $(U\times(0,\infty),d\tilde{\alpha})$ with a neighborhood of $\Gamma_H \times \{1\}$ in its symplectization $(\Gamma_H\times (0,\infty), d(s\tilde{\alpha}|_{\Gamma_H}))$, where $s$ the $(0,\infty)$-coordinate. This identification does not preserve the fiber of the projection in the second coordinates $h$ and $s$ of $U\times (0,\infty)$ and $\Gamma_H \times (0,\infty)$, respectively. Nonetheless, by composing this identification with the extension of the strict contactomorphism from $(\Gamma_H,\tilde{\alpha}|_{\Gamma_H})$ to $(U,\alpha)$ along the radial directions $s$ and $r$, of the symplectization $(\Gamma_H\times (0,\infty), d(s\tilde{\alpha}|_{\Gamma_H}))$ and $(U\times (0,\infty),d(r\alpha))$, respectively, we obtain a symplectomorphism between a neighborhood of $\Gamma_H$ in $(U\times (0,\infty), d\tilde{\alpha})$ and neighborhood of $U\times \{1\}$ in the symplectization of $(U,\alpha)$. In particular, this symplectomorphism sends hypersurfaces of contact type given by the graphs of Hamiltonians in $U$ that are $C^0$-close to $H$ and satisfy condition \eqref{eq: Hamilt cond so graph is contact} to hypersurfaces of contact type in the symplectization of $(U,\alpha)$.

Consider $U_+$ and $U_s$ the images of the graphs of $H_{+}$ and $H_s$ under the symplectomorphism, respectively. Recall we assumed $H_+$  and $H_s$ to be $C^0$-close to $H$ satisfying the same condition \eqref{eq: Hamilt cond so graph is contact}, hence the graphs $\Gamma_{H_+}$ and $\Gamma_{H_s}$ are hypersurfaces of contact type in $(U\times(0,\infty),d\tilde \alpha)$. Therefore, the graphs $\Gamma_{H_+}$ and $\Gamma_{H_s}$ are also of contact type in the symplectization of $(U,\alpha)$. Moreover $U_+$ lies, in the obvious sense, above $U$, and $U_s$ and below $U_+$ (here we are identifying $U$ with $U\times \{1\}$ in its symplectization). Observe that $U_-$ is the image of the graph $\Gamma_{H_-}$ and lies below $U$. The symplectization of $(U,\alpha)$ embeds in the completion $\widehat{W}$ of $W$, and therefore so do $U_{\pm}$ as hypersurfaces of restricted contact type. From conditions ($\text{H}_+$) and $(\text{H}_-3)$, the hypersurfaces $U_{+}$ and $U_s$ extend to closed contact hypersurfaces $N_+$ and $N_s$ of restricted contact type in $\widehat{W}$, with $N_+$ lying above $N$ and $N_s$, and $N_-$ lying below $N$. Hence, the pull back of $r\alpha$ to $N_+$ and $N_s$ is give by $\alpha_+ = f_+ \alpha$ and $\alpha_s = f_s\alpha$, where $f_+ \geq1 $ and $f_s\leq f_+$ and $f_- \leq1$. The extensions themselves are immaterial for our purposes. Furthermore, by construction, the regions $W_{+}$ and $W_s$ below $N_+$ and $N_s$ respectively are Liouville fillings and have the same topological homotopy type as $W$. The filings $W_{+}$, $W$, and $W_s$ intersect along $\gamma$, but this intersection can be eliminated by a small perturbation, at the cost of having an arbitrarily small change in the action filtration in symplectic homology; we may assume such a perturbation has been performed whenever necessary. In particular, after this perturbation $f+>1$, $f_s \leq f_+$ and $f_-<1$.

To summarize, we have constructed embeddings $N_+$,$N_s$ and $N_-$ in $\widehat{W}$, with $N_+$ lying above $N$ and $N_s$ below $N_+$, while $N_- = N_0$ lies below $N$. All these contact manifolds are of restricted contact type, with fillings $W_+$, $W_s$, and $W_-$, respectively. The loops $\gamma$ are a Reeb orbit for all the contact manifolds, and in the neighborhood $U = B\times S^1$ of $\gamma$, the forms are given by $\alpha_+ = \lambda_0 + H_+dt$, $\alpha_s = \lambda_0 + H_sdt$, and $\alpha_- = \lambda_0 + H_-dt$, respectively. 

As pointed out above, the Liouville domains $W_- \subset W  \subset W_+$ have the same homotopy type. In particular, the free homotopy class of the loops $\gamma^k$ in each of these domains can be identified. We denote this common class by $\mathfrak{f}^k$.

\textbf{Step 3.} \textit{Direct sum decomposition of symplectic homology:} We now focus on the class of Liouville domains $(W,\omega = d\lambda)$ such that the boundary $(N,\alpha)$  has an open subset $U$ that can be identified with $(B \times S^1, \lambda_0 + \tilde{H}dt)$, where $B = B_R(0)$, and $\tilde{H} : B\times S^1 \rightarrow (0,\infty)$ is a Hamiltonian with $\tilde{H}_t(0) = c$ a strict maximum for all $t\in S^1$, constant in $Y \times S^1$, where $Y = B_R \setminus B_{R/2}$ with $a:=\tilde{H}|_{Y\times S^1}$. Notice that $\gamma(t) = (0,ct)$ is a periodic Reeb orbit with action $c= \mathcal{A}(\gamma) $. Recall from \eqref{eq: Reeb vf in nbh} that Reeb vector field $R_{\alpha}$  of $\alpha$ in $U$ is given by 
$$R_{\alpha} = \frac{1}{\tilde{H} - \lambda_0(X_{\tilde{H}})}(\partial_t -X_{\tilde{H}}).$$
Hence, there is a one-to-one correspondence between the contractible $k$-periodic orbits of $-\tilde{H}$ in $B$, and the Reeb orbits in $U$ homotopic to $\gamma^k$. Moreover, the Hamiltonian actions correspond to the actions (period) of the associated Reeb orbits. 

Let $k\in \mathbb{N}$, $K$ be a semi-admissible Hamiltonian with $\text{slope}(K)>kc$, such that $r_{\max}(K)$ is arbitrarily close to $1$, and $a_K(kc)$ arbitrarily close to $kc$; see \eqref{eq: bilipchitz}, \eqref{eq:inequality action function} and \eqref{eq: A_h(t)convergetot}, and $I$ a interval around $kc$ sufficiently small. Consider the splitting of the Floer complex as modules
 \begin{equation} \label{eq: decomp in chain level}
     \text{CF}^I_*(K) = \text{CF}^I_*(K,U) \oplus \text{CF}^I_*(K;N\setminus U),
 \end{equation}
 where the first summand is generated by the one periodic orbits of $K$ associated to Reeb orbits in $U$ with action in $I$, and the second summand is spanned by all the remaining associated Reeb orbits with action in $I$ (since $kc>0$ and $I$ is sufficiently small, we do not need to consider orbits of $K$ coming from the filling $W$ of $(N,\alpha)$). 
 
\begin{lemma}[cf. \cite{ginzburg2010conley}, Lemma 7.1] \label{lemma: decomposition independe of k proof theorem}
There exists a constant $\epsilon_U>0$, depending only on $U$, such that \eqref{eq: decomp in chain level} yields a decomposition of chain complexes and induces a decomposition of symplectic homology 
\begin{equation} \label{eq:Ham decomp theorem}
    \text{SH}^I_*(W) = \text{HF}^I_*(K,U) \oplus \text{HF}^I_*(K, N \setminus U),
\end{equation}
for any Hamiltonian $K$ as described above (depending on $k$), and $I$ any action window around $kc$ with end points outside of the spectrum $\mathcal{S}(\alpha)$, not containing the points of $\{na \mid n\in \bb{N} \}$, with length $|I|< \epsilon_{U}$. Moreover, the decomposition \eqref{eq: new Ham decomp theorem} is independent of the filling $W$, and depends only on the boundary $(N,\alpha)$.
\end{lemma}
The key point of Lemma \ref{lemma: decomposition independe of k proof theorem} is that the constant $\epsilon_U$ is independent of $k$. In particular, the length of the action window $I$ can be taken uniform in $k$, which is essential for the following arguments.
\begin{proof}[Proof of Lemma \ref{lemma: decomposition independe of k proof theorem}] Let $K$ be a Hamiltonian as above, with $K=e(r)$ on $M \times [1,\infty)$. The Hamiltonian vector field of $K$ in $U\times[1,\infty)$ is given by 
$$X_K  = e'(r)R_\alpha =  \frac{e'(r)}{\tilde{H} - \lambda_0(X_{\tilde{H}})}(\partial_t - X_{\tilde{H}}).$$
Since $\tilde{H}$ is constant in $Y\times S^1$, it follows that $X_K$ is in the direction of $\partial_t$ in $Y\times S^1 \times [1,\infty) $. Fix an admissible almost complex structure $J \in \mathcal{J}(W)$. Let $u$ be any Floer cylinder for $K$ asymptotic to a Hamiltonian orbit coming from a Reeb orbit in $U$ in one end, and to a Hamiltonian orbit associated to a Reeb orbit that is contained in $M\setminus U$ at the other end (note that we only need to study this type of Floer cylinders since the possibility $u$ approaching a orbit in the shell $Y$ has been ruled out by assuming $\{nc \mid n\in \bb{N}\}$ does not intersect $I$). 
Then the projection $\pi(u)$, where $\pi: B\times S^1 \times  [1,\infty) \rightarrow B$ is a holomorphic map whose image crosses the shell $Y$. Crossing this shell requires some minimal energy $\tilde{\epsilon}_U$, depending only on $U$ (and not on $k$).

Now we claim that if the energy of the Floer cylinder $u$ is sufficiently small, then its image is contained in $N\times[1/2,\infty)$. Indeed, consider a collar neighborhood $Z$ of $N$ in $W$ such that $
(Z,\lambda|_Z) \cong (N \times [1/2,1],r\alpha)$. Under this identification, we view $N\times \{1/2\}$ as a submanifold of $W$. Since $K$ is an admissible Hamiltonian, in particular $K|_W = 0$. Hence, inside the collar neighborhood $N \times [1/2,1]$, the Floer cylinder satisfying \eqref{eq: floer eq} reduces to $J$ holomorphic curves. Thus, since both ends of $u$ lie in $N\times [1,\infty)$, if $u$ crosses the hypersurface $N\times \{1/2\}$, then it crosses the shell $N\times [1/2,1]$, and its energy is bounded from below by a constant $C = C(J)$. 

Now fix $\epsilon_U < \min\{ \tilde{\epsilon}_U,C\}$. Consider a $C^{\infty}$-small non-degenerate perturbation $\tilde{K}$ of $K$, and a $C^{\infty}$-small, compactly supported, generic periodic perturbation $\tilde{J}$ of $J$. Each $1$-periodic orbits of $K$ considered in \eqref{eq: decomp in chain level} split into finitely many non-degenerate $1$-periodic orbits of $\tilde{K}$ contained in a neighborhood of $B_{R/2} \times S^1 \times[1,\infty)$, and of $(N \setminus U)\times [1,\infty)$, respectively. It then follows from a suitable Gromov compactness theorem (see, e.g., \cite{fish2011target}) that any Floer cylinder $\tilde{u}$ of $\tilde{K}$ asymptotic to such perturbed orbits that either crosses the hypersurface $N\times \{1/2\}$ or that the projection $\pi(\tilde{u})$ crosses the shell $Y$ has energy bounded by below by $\epsilon_U$.

Hence, for action windows $|I|<\epsilon_U$ (independently of $k$), the decomposition \eqref{eq: decomp in chain level} is preserved by the Floer differential independently of $W$.    
\end{proof}

Note that the groups $HF^I_*(K,U)$ do not depend on the choice of $K$, provided that $K$ grows sufficiently fast (as required in the Lemma \ref{lemma: decomposition independe of k proof theorem}). In light if Lemma \ref{lemma: decomposition independe of k proof theorem}, we define $$SH^I_*(U,\lambda):=\lim HF^I_*(K,U) \text{  and  } SH^I_*(N \setminus U,\lambda):=\lim HF^I_*(K,N \setminus U),$$
where the limits are taken in the class of semi-admissible Hamiltonians. By taking the limit in \eqref{eq:Ham decomp theorem}, we obtain 
\begin{equation} \label{eq: new Ham decomp theorem}
    SH_*^I(W) = SH^I_*(U,\lambda) \oplus SH^I_*(N\setminus  U,\lambda).
\end{equation}
Moreover, in Lemma \ref{lemma: decomposition independe of k proof theorem}, the homology groups appearing on the decomposition may be restricted to a free homotopy class. In particular, \eqref{eq: new Ham decomp theorem} we have that
\begin{equation} \label{eq: decomposition replace Ham to symp}
    SH_*^I(W,\mathfrak{f}^k) = SH_*^I(U,\lambda,\mathfrak{f}^k)  \oplus SH^I_*(N\setminus  U,\lambda,\mathfrak{f}^k).
\end{equation}
Lastly, the decomposition \eqref{eq:Ham decomp theorem} is preserved  by the exact sequence \eqref{eq: long exa seq for Hamiltonian}, yielding, for convenient choices of $\epsilon>0,\delta>0$, the exact sequence
\begin{equation} \label{eq:exact seq bump function}
    SH_{*+1}^{(kc+\delta,kc+\epsilon)}(U,\lambda,\mathfrak{f}^k) \rightarrow SH_{*}^{(kc-\delta,kc+\delta)}(U,\lambda,\mathfrak{f}^k) \rightarrow SH_{*}^{(kc-\delta,kc+\epsilon)}(U,\lambda,\mathfrak{f}^k).
\end{equation} 
\begin{remark} \label{remark: remark lemma stays true with continuation}
By an argument similar to that used in the proof of Lemma \ref{lemma: decomposition independe of k proof theorem}, it follows that for any homotopy $\alpha_s$, $s \in [0,1]$ of contact forms, where each $\alpha_s$ belongs to the class described above, i.e., in the neighborhood $U = B\times S^1$, $\alpha_s$ can be written as $\lambda_0 + \tilde{H}^sdt$, where $\tilde{H}^s$ have strict local maximums at the $\{0\}\times S^1$ and are constant in $Y$, and any homotopy of Hamiltonians $K_s$, $s \in [0,1]$, each satisfying the assumption of Lemma \ref{lemma: decomposition independe of k proof theorem} for the respective $\alpha_s$, the continuation map $SH_*^I(U,\alpha_0) \rightarrow SH_*^I(U,\alpha_1)$ is an isomorphism for any $|I|<\epsilon_U$ satisfying the conditions in Lemma \ref{lemma: decomposition independe of k proof theorem} for all $\alpha_s$. 
\end{remark}

\textbf{Step 4.} \textit{Homological computations:} We now focus on computing $SH_*^{I}(U,\lambda,\mathfrak{f}^k)$, where $\tilde{H}$ is a $C^2$-small perturbation of a bump function; see \cite[Section 7]{ginzburg2010conley}. 

From the analysis of contractible $k$-periodic orbits of a $C^2$-small perturbation of a bump function $\tilde{H}$ in \cite{ginzburg2010conley}, together with the correspondence between Reeb orbits of $\lambda$ in $U$ and Hamiltonian orbits of $\tilde{H}$ in $B$, it follows that the $k$-periodic Reeb orbits in $U$, homotopic in $U$ to $\gamma^k$, with Conley-Zehnder indices $n-2, n -1 $ and $n$ (recall $B \subset \R^{2n-2}$), computed with respect to $\tau^k$, whose actions are in the interval $(kc -\delta,kc+ \epsilon)$, with 
$$\pi r^2<\epsilon, \text{ and }0<\delta \ll \epsilon,$$
where $r$ in the radius of the ball where $\tilde{H} \equiv \max(\tilde{H})$, and $\delta>0$ is sufficiently small (independent of $k$)
are described as follows:
\begin{itemize}
    \item There are no Reeb orbits with index $n-2$ and action in the range $(kc - \delta,kc+\epsilon)$,
    \item $\gamma^k$ is the only Reeb orbit of index $n-1$ with action in $(kc - \delta,kc + \delta)$,
    \item There exists a unique Reeb orbit $x$ of index $n$, with action in the range $(kc+\delta,kc+ \epsilon)$,
    
\end{itemize}
All the orbits listed above are non-degenerate Reeb orbits. Therefore, in order to compute $SH_{*}^{(kc - \delta,kc+ \delta)}(U,\lambda,\mathfrak{f}^k)$, we consider $K$ Hamiltonian as in Lemma \ref{lemma: decomposition independe of k proof theorem}, and perform a Morse-bott perturbation of the corresponding Hamiltonian orbits of $K$ of index $n-2,n-1$ and $n$. We denote these orbits by $y$ and $z$, corresponding to $\gamma^k$ and $x$, respectively. Under such perturbation, each orbits splits into two non-degenerate Hamiltonian orbits, denoted by $y_m , y_M$ and $z_m,z_M$, corresponding to the minimum and and maximum of a perfect Morse function defined on the circle of fixed points of the Hamiltonian diffeomorphism associated with $y$ and $z$, whose actions are arbitrarily close to the original ones, with Conley-Zehnder indices given by $\mu_{cz}(y_m) = n-1, \mu_{cz}(y_M) = n$ and $\mu_{cz}(z_m) = n, \mu_{cz}(z_M) = n+1$. Hence, it follows that 

\begin{equation} \label{eq: hom deg n z2}
    SH_{k\Delta + n}^{(kc-\delta, kc+\delta)}(U,\lambda, \mathfrak{f}^k) = \Z_2,
\end{equation}
generate by $y_M$ (recall that symplectic homology is computed with respect to the trivialization along orbits induced by the fixed trivialization of the anti-canonical bundle). The indices above, computed with respect to $\tau^k$, differ from those by $k\Delta$). Moreover, by deforming $\tilde{H}$ to a Hamiltonian with no periodic orbits of index $n-1$ and $n$, and applying Remark \ref{remark: remark lemma stays true with continuation}, a similar analysis of the orbits obtained from the Morse-Bott perturbation as above shows that 
\begin{equation} \label{eq: hom degre n zero}
    SH_{k\Delta + n}^{(kc-\delta,kc+\epsilon)}(U,\lambda,\mathfrak{f}^k)=0.
\end{equation}
Therefore, from the exact sequence \eqref{eq:exact seq bump function}, and \eqref{eq: hom deg n z2} and \eqref{eq: hom degre n zero}, we conclude that the map
\begin{equation} \label{eq:map surjetive}
    SH_{k\Delta + n+1}^{(kc+\delta,kc+\epsilon)}(U,\lambda,\mathfrak{f}^k) \rightarrow SH_{k\Delta + n}^{(kc-\delta,kc+\delta)}(U,\lambda,\mathfrak{f}^k) = \Z_2,
\end{equation}
is surjective. We now consider two forms $\lambda_0$ and $\lambda_1$ in $N$ in the above class, such that $\lambda_{i} = \alpha + H_{i}dt$ on $U$ with $H_1 \geq \tilde{H}\geq H_0$, and $\lambda_i = f_i \lambda$ outside of $U$, for $i=0,1$, where $f_1 \geq f_0$. As in step 2, we view $(M_i,\lambda_i)$ as contact submanifolds of restricted contact type in $\widehat{W}$, boundary of Liouville domains $W_{\pm}$, with $W_0 \subset W \subset W_1$. By the naturality of the transfer map, decomposition \eqref{eq:Ham decomp theorem} is preserved. Let $K_i$ semi-admissible Hamiltonians in $\widehat{W_i}$ (for the respective radial coordinates) be chosen suitably in the sense of Lemma \ref{lemma: decomposition independe of k proof theorem}. Then, from Remark \ref{remark: remark lemma stays true with continuation}, we conclude that
\begin{equation} \label{eq: iso to Z_2 proof therem}
    \Z_2 =  SH_{k\Delta + n}^{(kc-\delta,kc+\delta)}(U,\lambda_1,\mathfrak{f}^k) \rightarrow SH_{k\Delta + n}^{(kc-\delta,kc+\delta)}(U,\lambda_0,\mathfrak{f}^k) = \Z_2,
\end{equation}
is an isomorphism. 

\textbf{Step 5:} \textit{Final step.} In order to prove \eqref{eq:sym hom non zero}, and finish the proof of Theorem \ref{thm: SDM gene orbits}, it suffices to show that the transfer map
\begin{equation} \label{eq: map cant be zero}
    SH_{k\Delta + n+1}^{(kc +\delta_k,kc+\epsilon)}(W_{+} , \mathfrak{f}^k) \longrightarrow SH_{k\Delta + n+1}^{(kc + \delta_k,kc+\epsilon}) (W_{-},\mathfrak{f}^k),
\end{equation}
is nonzero.

Before proving that \eqref{eq: map cant be zero} is nonzero, we clarify how the various choices on the construction of $H_{\pm}$ depend on one another. 

The constant $c_{+}$ depends only on $U$ and $H$, and is fixed first. The parameter $\epsilon>0$, as well as the lower bound $k_0$ for $k$, depend on $U$, $H$, and the chosen constant $c_{+}$. Then, we choose the Hamiltonian $H_+$, which depends on $H$, $U$, $\epsilon$, and $c_+$. Lastly, for each $k>k_0$, we choose $H_-$, and therefore the constant $c_-$. The manner in which $\lambda_{\pm}$ are extended outside of $U$ is immaterial for our purposes. 

We now describe the choices of the constants in more detail. First, we pick $c_{+}$, with $c_+<c$ relatively close to $c$, e.g., $|c-c_+|<c/2$. Its is convenient to choose $c_{+} \in \bb{Q}c$, hence $c_{+} = p_{+}c/q_{+}$, for some integers $p_{+}$ and $q_{+}$.

Now we choose $\epsilon>0$ so that $\epsilon<\epsilon_U$ and is small relative to $|c-c_{+}|$, e.g., $\epsilon <  |c - c_+|/2,$ and lastly, $\epsilon< c/\max{\{q_-,q_+\}}$, so that the for all sufficiently large $k$, the intervals $I=(kc-\epsilon,kc + \epsilon)$ does no intersect the set $\{nc_{+}\mid n\in \Z \}$. 

We then assume $k$ is large enough so that $k|c-c_+| > \pi R^2$. We choose $H_+$ meeting the conditions ($\text{H}_+$) and \ref{eq: Hamilt cond so graph is contact}, and such that $\pi r^2 < \epsilon$, where $r$ is the radius of the ball where $H_+ \equiv \max{H_+}$. Note that this function is chosen independently of $k$. 

Lastly, for each $k$ sufficiently large, we choose $H_-$ as follows. By Proposition \ref{prop: geome characterization of SDM}, consider a sufficiently large $i$ and a loop of Hamiltonians $\eta_i = \eta$ such that
\begin{itemize}
    \item $\eta^t = \eta^t_i$, $t \in S^1$, is a loop of Hamiltonian diffeomorphisms fixing $0$,
    \item There exists a symplectic basis $\Xi = \Xi^i$ of $\R^{2n-2}$,
\end{itemize}
for which the Hamiltonian $K = K^i$, generating the flow of $(\eta^t)^{-1} \circ \varphi_H^t$ has a strict local maximum at $0$ and satisfies
\begin{equation} \label{eq: K function choice of F}
    \max k||d^2(K_t)_x||_{\Xi}<\pi,
\end{equation}
for all $x$ in a small neighborhood $V$ of the origin with $V \subset B_r$. Recall the ball was fixed with respect to the standard metric in $R^{2n-2}$, and is unrelated to the basis $\Xi$. Moreover, the loop $\eta$ has identity linearization at the origin, i.e., $d\eta^t_0 = I$, for all $t \in S^1$, and is contractible to the constant loop equal to the identity map in the
class of loops with identity linearization at $0$. Let $\eta_s$ be a homotopy from the loop $\eta$ to the constant loop equal to the identity such that $d\eta^t_s \equiv I$, for all $s,t$, and a family of one-periodic Hamiltonians $G^s_t$ generating $\eta_s^t$, normalized by $G^s_t(0) = 0$.  The condition $(d \eta^t_s)_0 = 0$ is equivalent to $d^2(G_t^s)_0 = 0$. 

Consider a bump function $F$ with respect to the basis $\Xi$. We only require $F$ to be constant outside of $V$ rather than compactly supported in $V$, such that
\begin{itemize}
    \item $F(0) = c$, and $F\leq K$, and by \eqref{eq: K function choice of F},
    \begin{equation} \label{eq: conley zehnd index is n-1}
        k||d^2F_0||_\Xi < \pi.
    \end{equation}
\end{itemize}
Furthermore, by choosing $\min F$ sufficiently small and using the condition $d^2(G^s_t)_0$, we can ensure that 
$$F^s:= G^s\#F \leq H_+, $$
for all $s$. Notice that $F^s$ does not need to be a bump function with respect to the standard coordinates of $B$. $F^s$ is a homotopy from $H_- := G^0 \#F \leq G^0\#K = H \leq H_+$, with $F^s(p) = c$, and $F^s \leq H_+$ for all $s$. We assume that $G^s_t \equiv0$ outside of a sufficiently small neighborhood of $0$. Consequently, the Hamiltonians $F^s$ are constant equal to $c_-=\min F$ on $Y$. We assume that $c_-$ satisfies similar conditions as $c_+$, so that Lemma \ref{lemma: decomposition independe of k proof theorem} applies. Since the change of basis from $\Xi$ to the canonical symplectic basis of $\R^{2n-2}$ is a symplectic linear map (in particular a linear map), then we make these choices of Hamiltonians so that condition \eqref{eq: Hamilt cond so graph is contact} holds. It follows from \eqref{eq: conley zehnd index is n-1} that the Conley-Zehnder index of $\gamma^k$ for $\lambda_{H_-}$ with respect to $\tau^k$ is $n-1$. 

Lastly, we take $\delta_k$ sufficiently small depending on $V$ (in particular $\delta_k$ may decrease with $k$). For more details about these choices, see \cite{ginzburg2010conley,ginzburg2009action, ginzburg2013closed}.

Now, back to \eqref{eq: map cant be zero}, consider the commutative diagram 
\begin{equation} \label{diag: diagram lambda to lambdaG}
    \begin{tikzcd}[row sep=large, column sep=large]
SH_*^{I}(U,\lambda{+},\mathfrak{f}^k)
\arrow[dr]
\arrow[d]
&
\\
SH_*^{I}(U,\lambda{-},\mathfrak{f}^k)
\arrow[r,"\cong"]
&
SH_*^{I}(U,\lambda_{F},\mathfrak{f}^k)
,
\end{tikzcd}
\end{equation}
where $\lambda_F = \lambda_0 +Fdt$, and $|I|<\epsilon_U$. From Remark \ref{remark: remark lemma stays true with continuation}, together with property ($\text{H}_-3$)) of $H_-$, the horizontal arrow in \eqref{diag: diagram lambda to lambdaG} is an isomorphism. By \eqref{eq: iso to Z_2 proof therem}, taking $I = (kc-\delta,kc+\delta)$ and degree $*=k\Delta + n$, the diagonal arrow in \eqref{diag: diagram lambda to lambdaG} is an isomorphism.
Hence, we have that 
\begin{equation} \label{eq: iso to Z_2 proof therem new}
    \Z_2 =  SH_{k\Delta + n}^{(kc-\delta,kc+\delta)}(U,\lambda_+,\mathfrak{f}^k) \rightarrow SH_{k\Delta + n}^{(kc-\delta,kc+\delta)}(U,\lambda_-,\mathfrak{f}^k) = \Z_2,
\end{equation}
is an isomorphism. Combining the observations above with \eqref{eq:map surjetive}, and using that the transfer map preserves the corresponding summand in the decomposition \eqref{eq: decomposition replace Ham to symp}, we conclude that the composition
\begin{equation} \label{eq: map cant be zero deg n}
   SH_{k\Delta + n+1}^{(kc +\delta_k,kc+\epsilon)}(W_{+} , \mathfrak{f}^k) \longrightarrow  SH_{k\Delta + n}^{(kc -\delta_k,kc+\delta)}(W_{+} , \mathfrak{f}^k) \longrightarrow SH_{k\Delta + n}^{(kc - \delta_k,kc+\delta}) (W_{-},\mathfrak{f}^k),
\end{equation}
where the left map is the connecting homomorphism in the exact sequence \eqref{eq:exact seq bump function}, and the right map is the transfer map,
is a nonzero map.

Therefore, by the naturality of the transfer map, the following diagram commutes

\[
\begin{tikzcd}
SH_{k\Delta + n+1}^{(kc +\delta_k,kc + \epsilon)}(W_+,\mathfrak{f}^k) \arrow[r] \arrow[d] & SH_{k\Delta + n}^{(kc -\delta_k,kc + \delta)}(W_+,\mathfrak{f}^k) \arrow[d] \\
SH_{k\Delta + n+1}^{(kc +\delta_k,kc + \epsilon)}(W_-,\mathfrak{f}^k) \arrow[r] & SH_{k\Delta + n}^{(kc -\delta_k,kc + \delta)}(W_-,\mathfrak{f}^k)
\end{tikzcd}
\]
and we conclude \eqref{eq: map cant be zero} is nonzero. This completes the proof.

\end{proof}

\subsection{Proof of Theorem \ref{thm: SDM generation on a covering}}
\begin{proof} We assume throughout the proof that all the Reeb orbits are isolated, or equivalently, that for any $L>0$, the set of Reeb orbits with action less than $L$ is finite, because otherwise \eqref{eq:growth in corollary} holds automatically. 

Assume $\gamma = z^m$, where $z$ is a simple Reeb orbit and $m\geq 1$. Consider the homology class $[z] = a\zeta \in H_1(W;\Z)/\text{Tor}$, where $\zeta$ is a primitive class in $H_1(W;\Z)/\text{Tor}$, $a\in\mathbb{\Z}$, and set $l:=ma$.  Since $\zeta$ is a primitive element of the free $\Z$-module $H_1(W;\Z)/\text{Tor}$, there is a decomposition
\begin{equation} \label{eq: dec modules}
    H_1(W;\Z)/\text{Tor} = \Z\langle\zeta \rangle \oplus A,
\end{equation} 
where $A$ is a free $\Z$-module. Consider the map $\pi_1(W) \rightarrow \Z\langle\zeta\rangle/ l\Z \langle \zeta\rangle$, given by the composition of the Hurewicz map with the projection along the first factor of \eqref{eq: dec modules}. The kernel of this map contains the homotopy class of $\gamma$, has index $l$ in $\pi_1(W)$ and determines a $l$-sheeted covering space $\Phi : W' \rightarrow W$. Let $\lambda' = \Phi^{*}(\lambda)$, so $(W',\lambda')$ is a Liouville domain whose boundary is a contact manifold $(\Sigma',\alpha')$. Note that $\Phi$ restricts to a map between the boundaries satisfying $\alpha' = \Phi^*(\alpha)$.

Choose $\gamma'$ a lift of $\gamma$ to $W'$. This is a simple closed curve representing a primitive homology class $\eta'\in H_1(W';\Z)/\text{Tor}$. Therefore, $\gamma'$ is a simple Reeb orbit. We will verify that $\gamma'$ is also a Reeb SDM for $(N',\alpha')$. Consider the local model of $z$ given by Lemma \ref{Lemma: local model reeb orbits}, where $U= B\times \R/c\Z$ with $\alpha = \lambda_0 + Hdt$, $H :B\times \R/c\Z  \rightarrow \R$, and $c = \mathcal{A}(z)$. Because the orbit $\gamma = z^m$ is a Reeb SDM, and $\varphi_{-H}$ is the return map of the simple Reeb orbit $z$, by Definition \ref{def: SDM}, the origin of $B$, seen as a constant orbit and denoted by $x$, is an SDM for $\varphi_{-H^{(m)}}= \varphi^m_{-H}$. This is the time one return map of $\gamma$, where $H^{(m)}(w,t) = mH(w,mt)$ for $(w,t)\in B\times S^1$. Moreover, by Remark \ref{rmk: iterate of a SDM is a SDM}, since $z$ and all its iterates are isolated Reeb orbits, $x^a$ is an SDM for $\varphi_{-H^{(l)}} = \varphi^a_{-H^{(m)}}$. The local model for $\gamma'$ in $(N',\alpha')$ from Lemma \ref{Lemma: local model reeb orbits} is given by $U' = B\times \R/lc\Z$, with contact form $\alpha' = \lambda_0 + H^{(l)}dt$. In particular, the return map of $\gamma'$ is $\varphi_{-H^{(l)}}$, and hence by Definition \ref{def: SDM} and the observation above, $\gamma'$ is a simple Reeb SDM for $(N',\alpha')$ representing a primitive non-torsion homology class in $H_1(W';\Z)$.

The action of $\gamma'$ is $\mathcal{A}(\gamma')=l\mathcal{A}(z)=lc$ and we denote its mean index by $\Delta$. By Theorem \ref{thm: SDM gene orbits} and \eqref{eq:local symp hom support}, for all sufficiently large prime number $p$, there exists a Reeb orbit $\gamma'_p$ of $(\Sigma',\alpha')$ with 
\begin{equation}\label{eq: action index inequality}
    \mathcal{A}(\gamma'_p)\in (plc+\delta_p,plc+\epsilon),
\end{equation}
and
\begin{align} \label{eq: mean index inequality}
    p\Delta+1\leq \Delta(\gamma'_k)\leq p\Delta+2n,
\end{align}
and representing a homology class $p[\gamma']=p\eta'\in H_1(W';\Z)$. Consider $\gamma_p=\Phi(\gamma'_p)$ to be the corresponding projected Reeb orbit. Note that each $\gamma_p$ has the same action and mean index as those of $\gamma'_p$. We will show below that the orbits $\gamma'_p$ are simple for all large primes $p$. Then, since $\Phi$ is a $l-$sheeted covering, for each large prime $p$, at most $l$ of the orbits $\gamma'_j$ satisfy $im(\Phi(\gamma'_j))=im(\gamma_p)$. The growth rate for $(\Sigma,\alpha)$ then follows.

Finally, we show that $\gamma'_p$ are geometrically distinct for all large primes $p$. For each prime $p$, the homology class $[\gamma'_p]/p=\eta'\in H_1(W';\Z)$ is primitive. Therefore, whenever $\gamma'_p$ is not simple, it is of the form $\gamma'_p=y^p$, where $[y]=\eta'\in H_1(W';\Z)/\text{Tor}$ is a primitive class and $y$ is a simple Reeb orbit. Consider $\mathcal{W}$ to be the set of simple Reeb orbits $y$ such that $\gamma_p'=y^p$ for some prime $p$, and note that by \eqref{eq: action index inequality}, for $y\in \mathcal{W}$ we have that $\mathcal{A}(y)\leq lc +\epsilon$. If $\mathcal{W}$ were infinite, by the period bound and compactness of $\Sigma'$ it would follow that $\Sigma'$ has a non-isolated Reeb orbit, and so would $\Sigma$. Hence $\mathcal{W}$ is finite. It follows by the mean index inequality above that $\gamma_p'$ must be simple for all large $p$. Otherwise, there would exist $y_0\in \mathcal{W}$ such that $y_0^p=\gamma'_p$ for infinitely many prime numbers $p$. The mean index inequalities \eqref{eq: mean index inequality} imply that 
\[\Delta+\frac{1}{p}\leq \Delta(y_0)\leq \Delta+\frac{2n}{p}\]
for infinitely many primes $p$.  Therefore, $\Delta(y_0)=\Delta$, but this contradicts $\Delta(\gamma'_p)=p\Delta(y_0)\geq p\Delta +1$ from \eqref{eq: mean index inequality}. Therefore, for every large prime $p$ the Reeb orbit $\gamma_p'$ is simple, which in turn finishes the proof of Theorem \ref{thm: SDM generation on a covering}.
\end{proof}

\subsection{Proof of Theorem \ref{thm: main_result}}
In this subsection we provide a proof of Theorem \ref{thm: main_result} relying on the machinery developed in the previous sections. 

Recall that the connected components of $\Lambda M$ are in bijection with conjugacy classes of $\pi_1(M)$. There is a natural surjection from this set to the abelianization of $\pi_1(M)$, which we identify as $H_1(M;\Z)$ via the Hurewicz homomorphism. Therefore, for a class $\eta\in H_1(M;\Z)$, we denote by $\Lambda_\eta M$ the union of all components of the free loop space corresponding to conjugacy classes of $\pi_1(M)$ that map to $\eta$ under the map described above. This yields the $H_1(M;\Z)$-decomposition of the free loop space
\begin{equation} \label{eq: H_1 splitting}
    \Lambda M = \bigsqcup_{\eta\in H_1(M;\Z)}\Lambda_\eta M,
\end{equation}
which is preserved by the Chas-Sullivan product of the free loop space homology, that is, for $\eta_1,\eta_2 \in H_1(M;\Z)$, the product is a map

\[H_n(\Lambda_{\eta_1} M) \otimes H_n(\Lambda_{\eta_2} M) \to H_n(\Lambda_{\eta_1+\eta_2} M). \]

\begin{proof}[Proof of Theorem \ref{thm: main_result}]
Let $\Sigma \subset T^*M$ be a fiberwise star-shaped hypersurface and $W\subset T^*M$ a Liouville subdomain bounding $\Sigma$, that is, $\partial W=\Sigma$. Also, let $\eta\in H_1(M;\Z)$ be a non-torsion homology class and $\beta\in H_k(\Lambda_\eta M)$ an $H_1(M)$-infinite non-nilpotent class. We assume throughout the proof that all the Reeb orbits are isolated, or equivalently, that for any $L>0$, the set of Reeb orbits with action less than $L$ is finite, because otherwise \eqref{eq:growth} holds automatically.

Consider the non-zero classes $\beta^m\in H_{mk-(m-1)n}(\Lambda_{m\eta}M)$, for $m\geq 1$. Viterbo's theorem provides an isomorphism of $\bb{Z}_2$-algebras
\begin{equation*} 
    H_*(\Lambda M)\to SH_*(W),
\end{equation*}
preserving the $H_1(M;\Z)$-splitting \eqref{eq: H_1 splitting}, where the right-hand side is endowed with the pair-of-pants product (see subsections \ref{subsec: cont and sym homology}, \ref{subsec: SH of a cotangent bdle} and \cite{abbondandolo2010floer}). This way, we can consider the classes $\beta^m$ in $SH_{mk-(m-1)n}(W;m\eta)$.

By Proposition \ref{Lemma: Reeb orbits with non zero loc hom}, let $x_m$ to be a spectral selector for $\beta^m$, which represents the homology class $m\eta$, so that $\mathcal{A}(x_m)=c_{\beta^m}(W)$ and $SH_{mk-(m-1)n}(x_m)\neq 0$. Equation \eqref{eq:local symp hom support} implies that
    \[\Delta(x_m)\in [m(k-n),mk-(m-2)n-1].\]
Moreover, the sublinearity property \eqref{eq: sublinearity spectral invts} implies that
\begin{equation*} \label{eq:lin for proof}
    c_{\beta^m}(W) \leq m c_{\beta}(W).
\end{equation*} 

For each prime number $p$, we write $x_p=y^{a_p}$, where $y$ is a simple Reeb orbit and $a_p\in \mathbb{N}$. We have two possibilities. 

If for all sufficiently large prime $p$, $a_p$ is not divisible by $p$, then the simple Reeb orbits $y_p$ such that $x_p = y_p^{a_p}$ have homology classes $[y_p] = p\frac{\eta}{a_p}$  which are multiples of $p$, therefore they are geometrically distinct. The growth rate \eqref{eq:growth} follows.

Otherwise, there exist infinitely many prime numbers $p$ such that $p$ divides $a_{p}$. Write $a_{p} = p d_p$ for some $d_p \in \mathbb{N}$. Define $\mathcal{W}$ to be the set of simple Reeb orbits $y$ such that $x_p=y^{pd_p}$ for some prime number $p$ and $d_p \in \mathbb{N}$. Note that since $x_p$ represents the homology class $p\eta$, the integer $d_p$ is a divisor of $\eta\in H_1(M;\Z)$, and the orbits $y\in \mathcal{W}$ represent homology classes $[y]=\eta/d_p$. Moreover,
\begin{equation} \label{eq: mean index estimate}
    \Delta(y)\in [\frac{k-n}{d_p},\frac{k-n}{d_p} +\frac{2n-1}{pd_p}],
\end{equation}
and 
    \[\mathcal{A}(y)= \frac{\mathcal{A}(x_p)}{pd_p}=\frac{c_{\beta^p}(W)}{pd_p}\leq \frac{c_\beta(W)}{d_p} \leq c_{\beta}(W).\]
    
If $\mathcal{W}$ were infinite, by the compactness of $\Sigma$ there would exist a non-isolated Reeb orbit. Hence $\mathcal{W}$ must be finite. Therefore, there exists an element $y\in \mathcal{W}$ and a sequence of prime numbers $p_i$ going to infinity such that $x_{p_i} = y^{p_id}$, for some fixed $d$ divisor of $\eta$. 

By the mean index estimate \eqref{eq: mean index estimate}, we have that 
\[\Delta(y)\in [\frac{k-n}{d},\frac{k-n}{d}+\frac{2n-1}{dp_i}],\]
for every $i\in \mathbb{N}$, so $\Delta(y)=\frac{k-n}{d}$.
Since $\Delta(x_{p_i})=p_i(k-n)$ and $SH_{p_i(k-n)+n}(x_{p_i})\neq 0$, by Corollary \ref{cor: SDM recognition}, each $x_{p_i}=y^{p_id}$ is a Reeb SDM for all $p_i$  representing the non-torsion class $p_i\eta\in H_1(M;\Z)$.
In particular, since a SDM orbit is totally degenerate, it follows that $\Delta(x_{p_i})$ must be an even integer, which forces $k\in n+2\Z$.
We then apply Theorem \ref{thm: SDM generation on a covering} to finish the proof.
\end{proof}

\section{String topology and examples} \label{sec: string topology and examp}

The goal of this section is to present a few ways to verify whether a closed
manifold admits a non-nilpotent homology class of a connected component of its free loop space corresponding to a non-torsion element $\eta\in H_1(M;\Z)$, i.e., an $H_1(M)$-infinite non-nilpotent class, which is the input needed in Theorem \ref{thm: main_result}. We focus on
two main methods: the existence of a section of the free loop space fibration, as observed in \cite{pellegrini2022bangert}, and a generalization given by the existence
of a free loop space homology class living over the fundamental class, as
envisaged in \cite{barraud2024floer}. We also describe, in Proposition \ref{prop: Serre spectral sequence and fundamental class}, a degeneration criterion 
in terms of the Leray-Serre spectral sequence associated to the free loop space fibration. Finally, in Example \ref{ex: no S1 action} we present manifolds admitting an $H_1(M)$-infinite non-nilpotent class but not a non-trivial $S^1$-action. The results in this section are classical and purely topological in nature.

Throughout this section, we assume $M$ to be a closed connected smooth manifold of dimension $n$ and $\eta\in H_1(M;\Z)$ a non-torsion element. We denote the free loop space evaluation map by $ev_0:\Lambda M \to M$, where $ev_0(\gamma)=\gamma(0)$. The induced map in homology is an algebra homomorphism between the singular homology of the free loop space and that of $M$, endowed with the Chas-Sullivan and intersection product, respectively; see Section \ref{subsec: SH of a cotangent bdle} and \cite[Ch 7, Proposition 1.2]{latschev2015free}.

\begin{rmk}
The whole theory could be developed over any field $\mathbb{K}$ of characteristic zero, as long as one takes into account suitable local systems in Viterbo's isomorphism, as in \cite{abouzaid2013symplectic}. We work over $\Z_2$ to avoid further technicalities.
\end{rmk}

\subsection{Homology class living over the fundamental class and sections}

We say that a homology class $\beta\in H_n(\Lambda M)$ \textit{lives over the
fundamental class} of $M$ if $(ev_0)_*(\beta)=[M]$, where $[M]\in H_n(M)$ is
the fundamental class of $M$; compare \cite[Definition~2.3]{barraud2024floer}.

The Chas--Sullivan product is additive with respect to the $H_1(M;\Z)$-filtration
\eqref{eq: splitting of free loop space}, in the sense that it takes the form
\[
H_n(\Lambda_{\eta_1} M)\otimes H_n(\Lambda_{\eta_2} M)\to H_n(\Lambda_{\eta_1+\eta_2} M),
\]
for $\eta_1,\eta_2\in H_1(M;\Z)$. Since the intersection product satisfies
$[M]\cdot [M]=[M]$, we obtain:

\begin{proposition}\label{prop: living above the fundamental class}
Suppose that $M$ admits a class $\beta\in H_n(\Lambda_\eta M)$ that lives over
the fundamental class of $M$, where $\eta\in H_1(M;\Z)$ is non-torsion. Then
$\beta$ is an $H_1(M)$-infinite non-nilpotent class.
\end{proposition}

If there exists a section $s\colon M\to \Lambda_\eta M$ of $ev_0$, then
$\beta:=s_*([M])$ lives over the fundamental class of $M$; hence it is an $H_1(M)$-infinite non-nilpotent class in the free loop space homology of $M$ by Proposition \ref{prop: living above the fundamental class}.

\begin{ex}
Among the manifolds admitting such a section are closed Lie groups $G$
with infinite fundamental group, since their free loop space fibration splits, up to homotopy, as $G\times \Omega G$. Then any loop representing a non-torsion
class $\eta\in H_1(G;\Z)$ gives rise to an $H_1(M)$-infinite non-nilpotent class. The same
is true for any closed manifold with an $H$-space structure.
\end{ex}

\begin{ex} 
Another class of examples whose free loop space fibration admits a suitable
section consists of the total space of a smooth, locally trivial, orientable $S^1$-fibration
\[S^1\hookrightarrow M\to B,\]
such that the fiber inclusion represents a non-torsion class
$\eta\in H_1(M;\Z)$. Indeed, a section of the free loop space fibration of $M$ is
given by a continuous choice of parametrization of the fibers. In particular,
any $S^1$-bundle with $\pi_2(B)=0$ has a free loop space section representing a class $\eta$.
\end{ex}

\begin{ex} \label{ex: mapping torus}

Let $V$ be a closed, connected, smooth manifold, and let $\phi\colon V\to V$
be a finite order diffeomorphism, i.e., $\phi^k=id_V$ for some $k\in\mathbb{N}$. The associated
mapping torus
\[M_\phi = V\times [0,1]/\sim,\]
where
\[(x,1)\sim (\phi(x),0), \qquad x\in V,\]
has fundamental group
\begin{equation} \label{eq: mapping torus fundamental group}
    \pi_1(M_\phi)=\pi_1(V)\rtimes_{\phi_*}\Z.
\end{equation}
Moreover, there is an isomorphism
\[H_1(M_\phi)\cong \mathrm{coker}\bigl(id-\phi_*\colon H_1(V;\Z)\to H_1(V;\Z)\bigr)\oplus \Z,\]
where the Hurewicz map is the identity on the $\Z$-factor; see
\cite[Example~2.48]{hatcher2005algebraic}. In this case, there is a section
$s\colon M_\phi\to \Lambda_{k[t]}M_\phi$ given by $s(x,t)=(x,kt)$. Hence,
there exists an $H_1(M)$-infinite non-nilpotent class with $\eta=k[t]$, where $t$ denotes a generator of the $\Z$ factor in \eqref{eq: mapping torus fundamental group}.

We note that the existence of such a section $s$ is not guaranteed when $\phi^k$ is only homotopic to the identity. This feature is exploited in Example \ref{ex: no S1 action}.
\end{ex}

The following result shows, in particular, that for manifolds with infinite
abelian fundamental group, it suffices to understand Proposition
\ref{prop: living above the fundamental class} for manifolds with free abelian
fundamental group, by passing to a finite-sheeted covering space. This follows
from the fact that a continuous map between spaces induces a map between their
free loop spaces. 

\begin{proposition}\label{prop: non-zero degree map}
Suppose that $M_1$ and $M_2$ are connected, closed, smooth manifolds of dimension $n$. Assume
that $M_1$ admits a class $\beta\in H_n(\Lambda_\eta M_1)$ that lives above the
fundamental class, and let $f\colon M_1\to M_2$ be a continuous map of non-zero
degree modulo $2$ such that $f_*(\eta)\in H_1(M_2;\Z)$ is non-torsion. Then $M_2$
has a $H_1(M_2)$-infinite non-nilpotent class.
\end{proposition}

\begin{rmk}
The class of manifolds $M$ that admit an $H_1(M)$-infinite non-nilpotent class is
closed under Cartesian product. More precisely, if $M_i$
are a closed smooth manifolds, $i=1,2$, and one of them admit an $H_1(M_i)$-infinite non-nilpotent class, then so does $M_1\times M_2$.
\end{rmk}

\subsection{The Gottlieb group}

The connected components of a free loop space fibration that admit a section are determined by the notion of Gottlieb group \cite{gottlieb1965certain, latschev2015free}. Consequently, the existence of suitable elements in this group - whose properties have been studied by homotopy-theoretic methods - provides a sufficient condition for the existence of an $H_1(M)$-infinite non-nilpotent class. We explore some applications below.

A \textit{cyclic homotopy} is a continuous map
\[h\colon M\times [0,1]\to M,\]
such that $h_0=h_1=id_M$. Given a point $x\in M$, the \textit{trace of $h$}
through $x$ is the loop
\[
t\in S^1 \longmapsto h(x,t).
\]
Note that $h$ induces a section of the free loop space fibration, and the
connected component of $\Lambda M$ in which this section takes values
determines an element of $\pi_1(M)$ (once a base point has been fixed).

\begin{definition}[\cite{gottlieb1965certain}]
The Gottlieb group of a based space $(M,x_0)$ is the subgroup
$G(M)\subset \pi_1(M;x_0)$ consisting of homotopy classes of loops that arise as
the trace through $x_0$ of some cyclic homotopy $h$.
\end{definition}

An element in the Gottlieb group whose Hurewicz map image is a non-torsion element $\eta\in H_1(M;\Z)$ yields a section of the free loop
space fibration, hence induces an $H_1(M)$-infinite non-nilpotent class.

\begin{proposition} \label{prop: Gottlieb}
    Suppose that $G(M)$ contains an element whose image in $H_1(M;\Z)$ under the
Hurewicz homomorphism is a non-torsion class $\eta$. Then there exists an $H_1(M)$-infinite non-nilpotent class in the free loop space homology of $M$.
\end{proposition}

Oprea proves in \cite[Theorem~6]{oprea1990homotopical} that whenever
$G(M)$ contains an element whose image under the Hurewicz map is non-torsion in
$H_1(M;\Z)$, as in Proposition \ref{prop: Gottlieb}, there exists a finite
cyclic covering $\bar{M}\to M$ such that $\bar{M}$ is homotopy equivalent to
 $Y\times S^1$ for some topological space $Y$. In particular, if $\bar{M}\to M$
has odd degree, then Proposition \ref{prop: non-zero degree map} implies the
existence of an $H_1(M)$-infinite non-nilpotent class for the free loop space of $M$. Nonetheless, Proposition \ref{prop: Gottlieb} remains
useful for producing examples.

We now explain how Proposition \ref{prop: Gottlieb} yields further applications.
Let $G$ be a compact Lie group acting continuously on $M$, and fix a basepoint
$x_0\in M$. The \textit{orbit map} $w\colon G\to M$ is defined by $w(g)=g\cdot x_0$.
\begin{corollary} \label{cor: G-actions}
   If the image of $w_*\colon \pi_1(G)\to \pi_1(M)$ contains an element that maps under the Hurewicz homomorphism to a non-torsion element $H_1(M;\Z)$, then $M$ admits an $H_1(M)$-infinite non-nilpotent class.
\end{corollary}
\begin{proof}
    The result follows from Proposition \ref{prop: Gottlieb} because from \cite[Property~3.2]{oprea2002bochner}, it follows that $im(w_*)\subset G(M)$. 
\end{proof}
A particular consequence of this corollary is that a homologically nontrivial $S^1$-action on $M$, as in \cite{pellegrini2022bangert}, implies the existence of an $H_1(M)$-infinite non-nilpotent class.
We observe that a conclusion similar to Corollary \ref{cor: G-actions} can be obtained by passing to a maximal torus of the Lie group  $G$. Indeed, the inclusion in $G$ induces a surjection between the fundamental groups, so the $G$-action restricts to a homologically non-trivial $S^1$-action.

Here is a concrete class of manifolds. Let $G\hookrightarrow M \rightarrow B$ be a principal $G$-bundle,
where $G$ is a compact, connected Lie group and $B$ is a closed smooth
manifold. There is a five-term exact sequence associated to Serre's spectral
sequence (see \cite[Chapter~5]{weibel1994introduction})
\[H_2(M;\Z)\to H_2(B;\Z)\to H_1(G;\Z)\to H_1(M;\Z)\to H_1(B;\Z)\to 0.\]
If $H_2(B;\Z)=0$ and the first Betti number of $G$ is positive, $b_1(G)>0$, then it follows
from Corollary \ref{cor: G-actions} that Theorem \ref{thm: main_result} applies
to $M$. Examples include principal $G$-bundles over $S^n$, $n\ge 3$, with
$b_1(G)>0$ and $S^1$-bundles over a $2$-torus.

 Particular examples of the above class are homogeneous spaces $H/G$, where $H$ is a
compact Lie group and $G\subset H$ is a closed connected subgroup. If, in
addition, $H_2(H/G;\Z)=0$ and $b_1(G)>0$, then Theorem \ref{thm: main_result}
applies to $H/G$. For instance, one can take $H=U(n)$ with $n\ge 2$ and
$G=U(n_1)\times SU(n_2)$ with $n_1+n_2\leq n$ and $n_1\geq 1$.

\subsection{Leray-Serre spectral sequence and a degeneration criterion}

    Consider the free loop space fibration
\[\Omega M\hookrightarrow \Lambda M\xrightarrow{ev_0} M,\]
where $\Omega M$ is the based loop space. This is a locally trivial fibration,
and there is an associated Leray-Serre spectral sequence converging to the homology of
$\Lambda M$. More precisely, there exists a spectral sequence of $\Z_2$-modules
\[\{E^r_{p,q},\, d^r\colon E^r_{p,q}\to E^r_{p-r,q+r-1}\}_{r\ge 2},
\qquad p,q\ge 0,\]
whose $E^2$-page is
\begin{equation}\label{eq: serre spec seq}
H_p\bigl(M;\mathcal{H}_q(\Omega M)\bigr),
\end{equation}
and which converges to $H_{p+q}(\Lambda M)$.

Here $\mathcal{H}_q(\Omega M)$ denotes the local system over $M$ with fiber
$H_q(\Omega M)$ and holonomy given by the conjugation action of $\pi_1(M)$ on
$H_q(\Omega M)$. More precisely, the local system is constructed by choosing a path
lifting function for the free loop space fibration; see
\cite[Chapter~1, Theorem~4.1]{latschev2015free}. Note that $\Omega M$ admits an
analogous decomposition to \eqref{eq: splitting of free loop space}, and hence
the spectral sequence \eqref{eq: serre spec seq} admits a corresponding
$H_1(M;\Z)$-splitting, which we denote by $E^r_{p,q}(\zeta)$ for
$\zeta\in H_1(M;\Z)$. Moreover,
\[H_n(\Lambda_\zeta M;\Z_2)\cong \bigoplus_{p+q=n} E^\infty_{p,q}(\zeta).\]
The following provides a criterion for the existence of a free loop space
homology class living over the fundamental class of $M$.

\begin{proposition}\label{prop: Serre spectral sequence and fundamental class}
Let $\eta\in H_1(M;\Z)$ be non-torsion. If there exists a nonzero class
\[
A\in H_n\bigl(M;\mathcal{H}_0(\Omega_\eta M)\bigr)=E^2_{n,0}(\eta)
\]
that survives to $E^\infty_{n,0}(\eta)$, then there is a class
$\beta\in H_n(\Lambda_\eta M)$ living above the fundamental class $[M]$.
Equivalently, this condition is that
\[
d^r(A)=0 \qquad \text{for all } r\ge 2.
\]
\end{proposition}

See \cite[Proposition~2.6]{barraud2024floer} for a proof. The relevant
differentials in Proposition \ref{prop: Serre spectral sequence and fundamental class} are
\begin{equation}\label{eq: Serre differentials}
d^r\colon E^r_{n,0}(\eta)\longrightarrow E^r_{n-r,r-1}(\eta),
\qquad r=2,\dots,n.
\end{equation}
We would like to find topological constraints on $M$ that force these
differentials to vanish. Note that
\[E^2_{n-r,r-1}(\eta)=H_{n-r}\bigl(M;\mathcal{H}_{r-1}(\Omega M)\bigr).\]

\begin{ex} \label{ex: aspherical mflds}
Proposition \ref{prop: Serre spectral sequence and fundamental class} applies to
a aspherical manifold, i.e., $M\simeq K(G,1)$ homotopy equivalent to an Eilenberg-MacLane space, provided that $G=\pi_1(M)$ has infinite abelianization and a non-vanishing class $A$ as above exists. Since $K(G,1)\simeq BG$, where $BG$ is the classifying
space of $G$, there is a homotopy equivalence
\[\Omega M\simeq \Omega BG\simeq G.\]
In particular, $H_\ell(\Omega M)=0$ for all $\ell\ge 1$, and hence the
differentials $d^r(A)$ in \eqref{eq: Serre differentials} vanish for $r\ge 2$.

We now investigate the existence of a nonzero class
$A\in H_n\bigl(M;\mathcal{H}_0(\Omega_\eta M)\bigr)$. Note that
$\mathcal{H}_0(\Omega M)$ is a rank-one local system over $M$ with fiber the
$\Z_2$-module
\[H_0(\Omega M)\cong \bigoplus_{\pi_1(M)}\Z_2,\]
and whose holonomy representation is given by the conjugation action of
$\pi_1(M)$ on itself. Thus,
\[\mathcal{H}_0(\Omega M)\cong \Z_2[\pi_1(M)].\]
Poincar\'e duality for singular homology with coefficients in a local system
implies that
\[H_n\bigl(M;\mathcal{H}_0(\Omega M)\bigr)\cong H^0\bigl(M;\mathcal{H}_0(\Omega M)\bigr),\]
where we are using coefficient in $\Z_2$ (more generally, one of the homologies would be twisted by the orientation local system); see \cite[Theorem~10.2]{spanier1993singular}
and \cite[p.~102]{davis2001lecture}. Moreover, there is an identification
\[H^0\bigl(M;\mathcal{H}_0(\Omega M)\bigr)\cong \Gamma\bigl(\mathcal{H}_0(\Omega M)\bigr)
\cong \Z_2\bigl[Z(\pi_1(M))\bigr],\]
where $\Gamma(\mathcal{H}_0(\Omega M))$ is the space of sections of the local system and $Z(\pi_1(M))$ denotes the center of $\pi_1(M)$; see \cite[p.~102]{davis2001lecture}. The above correspondences preserve the
$H_1(M;\Z)$-splitting \eqref{eq: splitting of free loop space}. In particular,
a nonzero class $A\in E^2_{n,0}(\eta)$ exists whenever $Z(\pi_1(M))$ contains a
nontrivial element that projects to a non-torsion element $\eta\in H_1(M;\Z)$.
\end{ex} 

\begin{ex} \label{ex: no S1 action}
In \cite[p.5]{cappell2013closed}, Cappell, Weinberger and Yan construct an example of a closed aspherical topological
manifold of dimension $6$ whose fundamental group has center $\Z$ and which does not admit a non-trivial $S^1$-action. Moreover, since a generator of the center has non-torsion image under the Hurewicz map in homology, Example \ref{ex: aspherical mflds} and Proposition \ref{prop: Serre spectral sequence and fundamental class} imply that such a manifold admits a free loop space homology class living over the fundamental class.

The manifold in consideration, denoted by $V_\phi$, is a mapping torus of a closed aspherical topological manifold $V$ of dimension 5 by a homeomorphism $\phi\colon V\to V$. The key feature of the construction is that, although $\phi^2$ is homotopic to $id_V$, the homeomorphism $\phi$ itself is not homotopic to an involution of $V$. In contrast with Example \ref{ex: mapping torus}, there is no obvious section of the free loop space fibration of $V_\phi$.

We remark that $V_\phi$ can in fact be endowed with a smooth structure, as observed in \cite[Remark~7.21]{weinberger2022variations}. Briefly, the reason is that $V$ is a closed PL-manifold of dimension 5, so it admits a smoothing \cite[Theorem~2.2]{milnordifferential}, and the homeomorphism $\phi$ can be chosen to be a diffeomorphism via surgery theory.

By taking products with closed hyperbolic manifolds, one concludes that in every dimension $\geq 6$ there exists a closed smooth manifold $M$ admitting an $H_1(M)$-infinite non-nilpotent class and yet no non-trivial $S^1$-action. In particular, Theorem \ref{thm: main_result} provides a growth rate for the number of Reeb orbits on star-shaped hypersurfaces in the cotangent bundle of these manifolds.
\end{ex}

\begin{rmk}
    We note that the notion of a \textit{manifold covered by diffeomorphisms}, as
    introduced in \cite[Definitions~6.22~and~6.29]{barraud2024floer}, provides
    another class of manifolds admitting a free loop space homology class living
    over the fundamental class. At present, the authors are not aware of concrete
    examples in this class that are not already discussed in the previous
    subsections.
\end{rmk}

\bibliographystyle{alpha}\bibliography{bibliography}

\end{document}